\documentclass[10pt, leqno]{article}

\usepackage{amsmath,amssymb,amsthm, epsfig}

\usepackage{hyperref}

\usepackage{amsmath}

\usepackage{amssymb}

\usepackage{hyperref}

\usepackage{color}

\usepackage{ulem}

\usepackage{epsfig}

\usepackage{dsfont}

\usepackage{mathrsfs}

\usepackage{wasysym}

\title{Some remarks on Campanato's Theorem and the Anisotropic Bessel Spaces}

\author{\it by \smallskip \\
H. Hajaiej \quad $\&$ \quad R. Leit\~ao
    \footnote{ \noindent \textsc{H. Hajaiej}.
\texttt{E-mail address: hhajaie@calstatela.edu}
    }
\footnote{
\noindent \textsc{R. Leit\~ao}.
Universidade Federal do Cear\'{a} - UFC. Department of Mathematics. Fortaleza - CE, Brazil - 60455-760.
\texttt{E-mail address: rleitao@mat.ufc.br}
}
                                                                                           }

\newlength{\hchng}
\newlength{\vchng}
\def \div {\mathrm{div}}

\def\Xint#1{\mathchoice
{\XXint\displaystyle\textstyle{#1}}
{\XXint\textstyle\scriptstyle{#1}}
{\XXint\scriptstyle\scriptscriptstyle{#1}}
{\XXint\scriptscriptstyle\scriptscriptstyle{#1}}
\!\int}
\def\XXint#1#2#3{{\setbox0=\hbox{$#1{#2#3}{\int}$ }
\vcenter{\hbox{$#2#3$ }}\kern-.6\wd0}}

\def\dashint{\Xint-}

\newtheorem{theorem}{Theorem}[section]

\newtheorem{lemma}[theorem]{Lemma}

\newtheorem{proposition}[theorem]{Proposition}

\newtheorem{corollary}[theorem]{Corollary}

\theoremstyle{definition}

\newtheorem{definition}[theorem]{Definition}

\theoremstyle{remark}

\newtheorem{remark}[theorem]{Remark}

\numberwithin{equation}{section}

\newcommand{\intav}[1]{\mathchoice {\mathop{\vrule width 6pt height 3 pt depth  -2.5pt
\kern -8pt \intop}\nolimits_{\kern -6pt#1}} {\mathop{\vrule width
5pt height 3  pt depth -2.6pt \kern -6pt \intop}\nolimits_{#1}}
{\mathop{\vrule width 5pt height 3 pt depth -2.6pt \kern -6pt
\intop}\nolimits_{#1}} {\mathop{\vrule width 5pt height 3 pt depth
-2.6pt \kern -6pt \intop}\nolimits_{#1}}}

\begin{document}
\maketitle

\begin{abstract}

In this paper, we establish an anisotropic version of Campanato's Theorem and show that the anisotropic Bessel spaces are continuously embedded in the function spaces $C^{\theta}_{\Vert \cdot \Vert}$-continuous. As an application of this embedding, we build fundamental solutions for a class of anisotropic fractional laplacians $\Delta^{\beta, \alpha}$.
\bigskip


 \noindent{\sc AMS Subject Classification MSC 2010}: 26A33; 35J70; 47G20.

\tableofcontents

\end{abstract}

\section{Introduction}

The goal of this paper is the study of a type of anisotropic fractional operators not included in previous works. In dimension 2, a family of anisotropic fractional laplacians was studied in several papers, \cite{Hajaiej-EP}, \cite{Hajaiej-SG} and \cite{Hajaiej-KMS}, and reference therein.  In this article, we will focus on the functional setting for this another class of anisotropic fractional laplacians, which constitutes an essential step in addressing PDEs involving these operators.\\

Given $\alpha > 0$, $\beta=\left( b_{1}, \dots, b_{n} \right)$, $b_{i} > 0$, and $1 < q < +\infty,$ we consider the anisotropic spaces $W_{\beta}^{\alpha, q}(\mathbb{R}^{n})$ and $H_{r_{\beta}}^{\alpha, q}(\mathbb{R}^{n})$ as follows

\begin{equation}
W_{\beta}^{\alpha, q}(\mathbb{R}^{n}) := \left  \lbrace u \in L^{q}(\mathbb{R}^{n}):   \displaystyle \dfrac{\vert u(z) - u(h) \vert^{q}}{\Vert z - h \Vert_{\beta}^{\frac{c_{\beta}}{q} + \alpha }} \in L^{q}\left( \mathbb{R}^{n} \times \mathbb{R}^{n} \right) \right \rbrace
\end{equation}
and
\begin{equation}
H_{r_{\beta}}^{\alpha, q}(\mathbb{R}^{n}) := \left  \lbrace u \in L^{q}(\mathbb{R}^{n}): \mathfrak{F}^{-1} \left[\left( 1 + r^{2}_{\beta}(z) \right)^{\alpha/2}  \mathfrak{F}(u) \right] \in L^{q}(\mathbb{R}^{n}) \right \rbrace,
\end{equation}
where $ \mathfrak{F}(f)$ denotes the Fourier transform of $f$, $c_{\beta} = \displaystyle \sum_{i=1}^{n} \frac{2}{b_{i}}$ is the anisotropic dimension,
$$
\Vert x - y \Vert_{\beta} = \left( \sum_{i=1}^{n} \vert x_{i} - y_{i} \vert^{b_{i}} \right)^{1/2}
$$
is the anisotropic quasi-distance, see \cite{CLU, LEITAO, SL}, and $r_{\beta}$ is the anisotropic distance that satisfies the Calder\' on-Besov-Lizorkin condition, namely,
\begin{eqnarray}
\label{CBL condition}
\sum_{i=1}^{n}\dfrac{x_{i}^{2}}{r_{\beta}^{\frac{4}{b^{\ast}_{i}}}} = 1, \quad \beta^{\ast} = \dfrac{c_{\beta} }{n}\beta, \quad \text{and} \quad c_{\beta^{\ast}} = n,
\end{eqnarray}
see \cite{BESOV-LIZORKIN}. The anisotropic Gagliardo space $W_{\beta}^{\alpha, q}(\mathbb{R}^{n})$ and the anisotropic Bessel space $H_{r_{\beta}}^{\alpha, q}(\mathbb{R}^{n})$ are endowed respectively with the natural norms

\begin{equation}
\Vert u \Vert_{W_{\beta}^{\alpha, q}(\mathbb{R}^{n})} := \left( \int_{\mathbb{R}^{n}} \vert u (z) \vert^{q} dz + \int_{\mathbb{R}^{n}} \int_{\mathbb{R}^{n}}  \dfrac{\vert u(z) - u(h) \vert^{q}}{\Vert z - h \Vert_{\beta}^{c_{\beta} + \alpha q}} dz \, dh \right)^{1/q}
\end{equation}
and
\begin{equation}
\Vert u \Vert_{H_{r_{\beta}}^{\alpha, q}(\mathbb{R}^{n})} := \left  \Vert \mathfrak{F}^{-1} \left[\left( 1 + r^{2}_{\beta}(z) \right)^{\alpha/2}  \mathfrak{F}(u) \right] \right \Vert_{q}.
\end{equation}

In the study of PDEs, it is fundamental to know the regularity of the solutions, see for example \cite{S, CS, CLU, ROS-OTON-SERRA 1, ROS-OTON-SERRA 2, LTY, CTU, SL, CKP}. This is where the continuous embedding in H\"{o}lder continuous spaces is crucial. In the isotropic case $b_{i} = 2$ and $c_{\beta} = n,$ the relation
\begin{eqnarray}
\label{isotrop relation}
W_{\beta}^{\alpha, q}(\mathbb{R}^{n}) = H_{r_{\beta}}^{\alpha, q}(\mathbb{R}^{n}), \quad 0 < \alpha < 1,
\end{eqnarray}
is well-known in the literature, see for example Theorem 7.63 in \cite{ADAMS}, Theorem 5 in \cite{STEIN}, and Theorem 27.3 in \cite{SAMKO}. An interesting and useful consequence of \eqref{isotrop relation} and the classical Campanato's Theorem, see Theorem 2.9 in \cite{E. GIUSTI} and Theorem 8.2 in \cite{NPV}, is the continuous embedding
\begin{eqnarray}
H_{r_{\beta}}^{\alpha, q}(\mathbb{R}^{n}) \hookrightarrow C^{\theta}_{\Vert \cdot \Vert_{\beta}}, \quad 0< \theta < 1,
\end{eqnarray}
where $C^{\theta}_{\Vert \cdot \Vert_{\beta}}$ is the set of functions $\theta$-H\" older continuous in the quasi-distance $\Vert \cdot \Vert_{\beta}$, $b_{i} = 2$, endowed with the norm
\begin{eqnarray}
\Vert u \Vert_{C^{\theta}_{\Vert \cdot \Vert_{\beta}}(\mathbb{R}^{n})} := \Vert f \Vert_{L^{\infty}(\mathbb{R}^{n})} + \displaystyle\sup_{x, y \in \mathbb{R}^{n}} \dfrac{\vert u(x) - u(y)\vert}{\Vert x - y \Vert_{\beta}^{\theta}}.
\end{eqnarray}
Here the term continuous embedding means that there exists a constant $C >0$ independent of $u$ such that
\begin{eqnarray}
\label{embedding}
\Vert u \Vert_{C^{\theta}_{\Vert \cdot \Vert_{\beta}}(\mathbb{R}^{n})} \leq C \Vert u \Vert_{H_{r_{\beta}}^{\alpha, q}(\mathbb{R}^{n})}.
\end{eqnarray}
In this paper, we establish an anisotropic version of \eqref{embedding} under the condition $c_{\beta} = n$. Our strategy is to first prove the classical Campanato's Theorem for the anisotropic case and then link that to the theory developed by Lizorkin in \cite{LIZORKIN 1, LIZORKIN 2, LIZORKIN 3} and the results obtained by the authors in \cite{LEITAO LP}. We shall prove the following:
\begin{theorem}
\label{main theorem}
Let $\beta = (b_{1}, \dots, b_{n}),$ and $b_{\max}=\displaystyle\max_{1\leq i\leq n}b_i.$ If $0 < \alpha < 2/b_{\max}$ and $\alpha q > c_{\beta}$, then $W_{\beta}^{\alpha, q}(\mathbb{R}^{n})$ is continuously embedded in $C^{\theta}_{\Vert \cdot \Vert_{\beta}}\left( \mathbb{R}^{n}\right)$ for
\begin{eqnarray}
\label{holder reg}
\theta = \alpha - \dfrac{c_{\beta}}{q}.
\end{eqnarray}
In particular, if $c_{\beta}=n$, $H_{r_{\beta}}^{\alpha, q}(\mathbb{R}^{n})$ is continuously embedded in $C^{\theta}_{\Vert \cdot \Vert_{\beta}}\left( \mathbb{R}^{n}\right)$.
\end{theorem}
In order to prove the first part of Theorem \ref{main theorem}, we adapt Campanato's approach to the geometry driven by the quasi-distance $\Vert \cdot \Vert_{\beta}$. Indeed, we show that $u \in W_{\beta}^{\alpha, q}(\mathbb{R}^{n})$ satisfies a \textit{decay of oscillation of its average in anisotropic balls}. Another essential step towards the $C_{\Vert \cdot \Vert_{\beta}}^{\theta}$-regularity for $u$ is \textit{Lebesgue differentiation theorem for anisotropic balls}. In the sequel, we comment with details on the strategies to reach our goals:

\begin{enumerate}
\item (Decay of oscillation of the average of $u$ in anisotropic balls) As in the seminal Campanato's technique, we show that if $c_{\beta} < \alpha \leq c_{\beta} + q$ and $0 < r < R$, there exists a constant $C=C(n,\beta, \alpha) > 0$ such that
\begin{eqnarray}
\label{decay of average}
\left | u_{x_{0},\sigma} - u_{x_{0},\rho} \right | \leq C \left[ u \right]_{q, \alpha q}\left( \dfrac{\sigma}{\rho}\right)^{\alpha}  \rho^{\theta}, \quad \text{for} \ 0 < r \leq \rho < \sigma \leq R,
\end{eqnarray}
where $u_{x_{0},r}:= \dfrac{1}{\left | E_{r}(x_{0}) \right |}\displaystyle \int_{E_{r}(x_{0})}u(x)\, dx$ is the average of $u$ in the ellipsoid $E_{r}(x_{0})$ centered at $x_{0}$ and with the $i$-th edge equal to $r^{2/b_{i}}$, see \cite{BESOV-LIZORKIN, CLU, SL, LEITAO}, and
\begin{eqnarray}
\left[ u \right]^{q}_{q, \alpha q}:= \sup_{x_{0}\in \mathbb{R}^{n}, r>0} r^{-\alpha q}\displaystyle \int_{E_{r}(x_{0})}\vert u(x) - u_{x_{0},r} \vert^{q} \, dx
\end{eqnarray}
is the anisotropic Campanato's semi-norm.
\item (Lebesgue differentiation theorem for anisotropic balls) Here we invoke an anisotropic covering lemma proved by Caffarelli-Calder\' on, see Lemma 3 in \cite{CC} or Lemma 3.6 in \cite{CLU}, to get
\begin{eqnarray}
\displaystyle \lim_{r \rightarrow 0} u_{x_{0}, r} = u(x_{0}) \quad \text{for almost every}\ x_{0} \in \mathbb{R}^{n}.
\end{eqnarray}
With 1 and 2 in mind, we prove the $C^{\theta}_{\Vert \cdot \Vert_{\beta}}$-regularity for $u$ as in the classical case $b_{i}=2$.
\end{enumerate}

The second part of Theorem \ref{main theorem} is a strong result as it connects the three anisotropic spaces $H_{r_{\beta}}^{\alpha, q}(\mathbb{R}^{n}),$ $B^{\beta}_{q}(\mathbb{R}^{n}),$ and $W_{\beta}^{\alpha, q}(\mathbb{R}^{n}) .$ The latter were studied separately. More precisely:
\begin{enumerate}
\item The space $H_{r_{\beta}}^{\alpha, q}(\mathbb{R}^{n})$ was studied by Lizorkin in \cite{LIZORKIN 1, LIZORKIN 2}.
\item The anisotropic Besov space $B^{\beta}_{q}(\mathbb{R}^{n})$ which consists of all functions $f$ in $L^{q}(\mathbb{R}^{n})$ such that
\begin{eqnarray*}
\int_{0}^{\infty} \int_{\mathbb{R}^{n}} \dfrac{\vert u (y_{1}, \dots, y_{i} + t, \dots, y_{n}) -  u (y_{1}, \dots, y_{i}, \dots, y_{n}) \vert^{q}}{t^{1 + qb_{i}}} \, dy \, dt < + \infty, \quad \text{for} \ i=1, \dots, n,
\end{eqnarray*}
see \cite{BESOV-I-N, LONG-TRIEBEL, SCHMEISSER-TRIEBEL}.
\item The space $W_{\beta}^{\alpha, q}(\mathbb{R}^{n}) $ was addressed by the authors in \cite{LEITAO LP}, see also \cite{LONG-TRIEBEL}.
\end{enumerate}
This connection is established as follows, in Lemma \ref{Main. Lemma} in Section 2. In \cite{LIZORKIN 2}, Lizorkin demonstrated, via hypersingular integrals, the anisotropic relation
\begin{eqnarray}
\label{LIZORKIN}
L^{\beta}_{q}(\mathbb{R}^{n}) = H_{r_{\beta^{}}}^{\frac{2n}{c_{\beta}}, q}(\mathbb{R}^{n}),
\end{eqnarray}
where $L^{\beta}_{q}(\mathbb{R}^{n})$ is the anisotropic fractional Sobolev space defined by $f$ in $L^{q}(\mathbb{R}^{n})$ such that
\begin{eqnarray}
 \mathfrak{F}^{-1}\left[ \left( 1 + x^{2}_{i}\right)^{b_{i}/2}\mathfrak{F}(f)\right] \in L^{q}(\mathbb{R}^{n}).
\end{eqnarray}
On the other hand, the second author showed that, see Lemma 2.8 in \cite{LEITAO LP},
$$
W_{\beta}^{\alpha, q}(\mathbb{R}^{n}) = B^{\beta(\alpha)}_{q}(\mathbb{R}^{n}), \quad \text{if} \ \beta(\alpha) = (\alpha b_{1} /2 ,\dots,  \alpha b_{n} /2) \ \text{and} \ 0 < \alpha < \dfrac{2}{b_{\max}},
$$
see also Theorem 1 in \cite{LONG-TRIEBEL}. Therefore, \eqref{embedding} is a consequence of the continuous embedding $$L^{\beta}_{q}(\mathbb{R}^{n}) \hookrightarrow B^{\beta}_{q}(\mathbb{R}^{n})$$
obtained by Lizorkin in \cite{LIZORKIN 3}. It is important to mention that for $q=2$ and $c_{\beta} = n$ Lizorkin proved that
\begin{eqnarray}
\label{Lizorkin relation}
H_{r_{\beta^{}}}^{\alpha, q}(\mathbb{R}^{n}) = L^{\beta(\alpha)}_{q}(\mathbb{R}^{n}) = B^{\beta(\alpha)}_{q}(\mathbb{R}^{n}) = W_{\beta}^{\alpha, q}(\mathbb{R}^{n}) \quad \text{for} \ 0 < \alpha < \dfrac{2}{b_{\max}}.
\end{eqnarray}
Moreover, by Proposition 2.11 in \cite{LEITAO LP}, the anisotropic Bessel spaces $H^{\alpha, 2}_{\beta}\left ( \mathbb{R}^{n} \right)$ studied in \cite{LEITAO LP} satisfy
\begin{eqnarray}
\label{Leitao relation}
W_{\beta}^{\alpha, 2}(\mathbb{R}^{n}) = H^{\alpha, 2}_{\beta} \left( \mathbb{R}^{n} \right), \quad 0 < \alpha < \dfrac{2}{b_{\max}},
\end{eqnarray}
for any $c_{\beta} > 0$ (including the case $c_{\beta} \neq n$). The equalities \eqref{Lizorkin relation} and \eqref{Leitao relation} indicate that the condition $c_{\beta} = n$ may be removed from the anisotropic embedding \eqref{embedding} if we consider the spaces $H^{\alpha, q}_{\beta}\left ( \mathbb{R}^{n} \right)$. We believe that a refinement of Lizorkin's approach may shed some light on this issue. We plan to treat this problem in a forthcoming paper. Another relevant remark is that $\Vert \cdot \Vert_{\beta}$ and $r_{\beta}$ are comparable and determine the same scaling, that is,
\begin{eqnarray}
\label{relation between Bessel Lizorkin and Bessel Leitao}
C^{-1}\Vert x \Vert_{\beta} \leq r_{\beta}(x) \leq C \Vert x \Vert_{\beta} \quad \text{and} \text \quad  r_{\beta}(T_{\beta, \kappa}x) = \kappa \, r_{\beta}(x),\  \Vert T_{\beta, \kappa}x \Vert_{\beta} = \kappa \, \Vert x \Vert_{\beta},
\end{eqnarray}
for all $x \in \mathbb{R}^{n}$, where $C>0$ is a constant depending on $n$ and $\beta$, $T_{\beta, \kappa}e_{i} = \kappa^{2/b_{i}} e_{i}$ is a linear map and $e_{i}$ is the i-th canonical vector. Hence, from \eqref{relation between Bessel Lizorkin and Bessel Leitao} and Fourier multipliers, see Remark 2.2.12 (ii) and Theorem 2.3.2 in \cite{BESSELFARKAS}, we infer that
\begin{eqnarray}
\label{equality Bessel Lizorkin and Bessel Leitao}
H_{r_{\beta^{}}}^{\alpha, q}\left( \mathbb{R}^{n} \right) = H^{\alpha, q}_{\beta}\left ( \mathbb{R}^{n} \right), \quad \text{if} \ c_{\beta} = n.
\end{eqnarray}
As an application of Theorem \ref{main theorem}, we use relation \eqref{equality Bessel Lizorkin and Bessel Leitao} to prove the existence of a fundamental solution at origin to the equation
\begin{eqnarray}
\label{anisotropic harmonic equation}
\Delta^{\beta, \alpha} u = 0, \quad \text{in} \ \mathbb{R}^{n} \setminus \left\lbrace 0 \right\rbrace,
\end{eqnarray}
where $\Delta^{\beta, \alpha}$ is the anisotropic fractional laplacian given by
$$
\Delta^{\beta, \alpha}u (x) = C_{\beta, \alpha} \int_{\mathbb{R}^{n}} \dfrac{u(x) - u(y)}{\Vert x - y \Vert_{\beta}^{c_{\beta}+2\alpha}} dy, \quad c_{\beta} = n, \quad C_{\beta, \alpha} = \dfrac{2}{b_{\max}} - \alpha,
$$
if $u$ is sufficiently smooth. More specifically, we show the existence of a viscosity solution of \eqref{anisotropic harmonic equation} which goes to infinity at the origin and is locally bounded in $\mathbb{R}^{n} \setminus \left \lbrace 0 \right \rbrace$ for $\alpha$ near $2/b_{\max}$. The study of the existence of fundamental solutions has an extensive literature:
\begin{enumerate}
\item \textbf{Local}: In an original work \cite{BOCHER}, B\^ocher in 1903 found the fundamental solutions for the Laplacian $\Delta$. In \cite{GSer}, Gilbarg and Serrin, via more modern theories, addressed the case quasi-linear. A few years later, in the series of papers \cite{Ser 1, Ser 2, Ser 3}, Serrin completed the study of general case $\text{div} \, \mathcal{A}(x,u, Du) = \mathcal{B}(x, u, Du)$, where the $p$-Laplacian $\Delta_{p}u= \div \, \left( \vert D u \vert^{p-2} D u \right)$, $p>0$, was the prototype operator. For fully nonlinear operators we refer Labutin's works \cite{LABUTIN 1, LABUTIN 2} in which a fundamental solution was reached for Pucci operators. For more general fully nonlinear equations, for instance Isaacs operator, Armstrong, Sirakov, and Smart in \cite{ASS 1, ASS 2} proved the existence of fundamental solutions and, as a consequence, they obtained Liouville-type results.
\item \textbf{Non-Local}: In \cite{FQ 1, FQ 2}, Felmer and Quaas treated of the existence of fundamental solutions for a class of non-local Issacs operators governed by a symmetric kernel. A Liouville-type theorem was also obtained. Inspired by the ideas presented in \cite{ASS 1, ASS 2}, Cao, Wu, and Wang in \cite{CWW} showed the existence of fundamental solutions for integro-differential operators like a fractional Laplacian. Recently in \cite{NPQUAAS}, Nornberg, Prazeres and Quaas study fundamental solutions for extremal fully nonlinear integral operators in conical domains.
\end{enumerate}
With the aim of finding anisotropic fundamental solutions of \eqref{anisotropic harmonic equation} we follow the same line of attack as in \cite{ASS 1, ASS 2, CWW} which is inspired by Perron's Method and comparison principle. In the isotropic case, see \cite{CWW}, the key step is to build, via a suitable maximum principle, a positive and $-\gamma$-homogeneous viscosity solution to the non-local Poisson equation
\begin{eqnarray}
\label{Poisson equation introduction}
\Delta^{\beta_{0},\alpha}u_{\gamma}(x) = \dfrac{1}{\Vert  x \Vert_{\beta_{0}}^{\gamma + 2\alpha}}, \quad \text{for all} \ x\in \mathbb{R}^{n} \setminus \left\lbrace 0 \right\rbrace, \quad \beta_{0}=(2, \dots, 2),
\end{eqnarray}
for $0 < \gamma < \gamma^{\ast}$, where $\gamma^{\ast}$ is the scaling exponent of $\Delta^{\beta_{0},\alpha}$ given by $\gamma^{\ast} = \sup\mathcal{I}_{\alpha}$, with $\mathcal{I}_{\alpha}$ being the set of $\gamma \in (0, +\infty)$ such that there exists a positive and $-\gamma$-homogeneous vsicosity supersolution $u_{\gamma}$ of
\begin{eqnarray}
\Delta^{\beta_{0},\alpha}u_{\gamma} \geq 0 , \quad \text{in} \ \mathbb{R}^{n} \setminus \left\lbrace 0 \right\rbrace.
\end{eqnarray}
Through Schauder estimates and stability properties, a standard normalization argument assures that $u = \lim_{\gamma \rightarrow \gamma^{\ast}} u_{\gamma}$ is a fundamental solution to $\Delta^{\beta_{0},\alpha}u = 0$. Besides, a comparison principle for $-\gamma$-homogeneous functions shows that $u$ is unique up to a constant. In \cite{CWW}, the ingredients that guarantee the construction of a viscosity solution of \eqref{Poisson equation introduction} are the existence of $-\gamma$-homogeneous barriers, a background on viscosity solutions theory, $C^{\theta}$ and $L^{q}$-estimates for strong solutions to the non-local Schr\" odinger equation and $C^{\theta}$-regularity for functions in Bessel spaces, see \cite{CS,DK,STEIN,NPV}. Naturally, the anisotropic structure of the kernel that governs $\Delta^{\beta, \alpha}$ requires an improvement of the techniques and appropriate versions of tools used in the isotropic case. In the sequel, we explain how these ingredients appear in the construction of solutions to an anisotropic version of the equation \eqref{Poisson equation introduction} and where we can find them. We also comment on the main difficulties we came across and how we overcame them.

\begin{enumerate}
\item(Barrier function) The first step towards the existence of solutions of \eqref{Poisson equation introduction} is to show that $0 < \gamma^{\ast} < n$. In the anisotropic case, this is not immediate. Given $r> 0$, denote by $\Theta^{\beta}_{r}$ the open ball of radius $r$ in the quasi-distance $\Vert \cdot \Vert_{\beta}$. In \cite{CWW}, the (isotropic) convexity of $u_{\gamma} = \Vert  x \Vert_{\beta_{0}}^{-\gamma}$ on the unitary sphere $\partial \Theta^{\beta_{0}}_{1}$ and the intrinsic scaling of $\Delta^{\beta_{0}, \alpha}$ suggest that $u_{\gamma}$ is the natural barrier for $\Delta^{\beta_{0}, \alpha}$, for $\gamma >0$ sufficiently small. In our case, radial functions are still the suitable barriers. However, in general, $\Vert \cdot \Vert^{-\gamma}_{\beta}$ is not convex on $\partial \Theta^{\beta_{}}_{1}$, in fact, it is not even $C^{2}$. We solve this issue by dilating (uniformly) the homogeneity degrees $b_{i}$ and by refining the construction of barriers made in \cite{LEITAO LP}. Our barrier is given by
\begin{eqnarray}
u_{\gamma} = \Vert x \Vert_{\mu \beta}^{-\gamma/\mu},
\end{eqnarray}
for some $\mu > 1$ sufficiently large. As in \cite{LEITAO LP}, the geometry determined by $\Vert \cdot \Vert_{\mu \beta}$ requires an adequate convexity of $u_{\gamma}$ (depending on $x$) on $\partial \Theta^{\mu \beta}_{1}$. Indeed, we get a universal estimate of the anisotropic convexity of $u_{\gamma}$ outside of \textit{anisotropic cones with universally small aperture}. From this estimate, we use a \textit{fine relation} between anisotropic cones and isotropic cones to prove that $\Delta^{\beta, \alpha}$ \textit{does not degenerate} as $\alpha \rightarrow 2/b_{\max}$ on $\partial \Theta^{\mu \beta}_{1}$, see Lemma \ref{Bar. func.} in Section \ref{sec Fundamental solution: ingredients}. This analysis enables us to access Caffarelli-Silvestre's approach, see \cite{CS}, to get $0 < \alpha_{0} < 2/b_{\max}$ depending on $\beta$ and $n$ such that
\begin{eqnarray}
\Delta^{\beta, \alpha} u_{\gamma} > 0, \quad \text{on} \ \partial\Theta^{\mu \beta}_{1},
\end{eqnarray}
for all $\alpha_{0} < \alpha < 2/b_{\max}$, which leads us to infer that $0 < \gamma^{\ast}$, see Lemma \ref{scaling lemma on the anisotropic sphere} in Section \ref{sec Fundamental solution: ingredients}. The bound $\gamma^{\ast} < n$ is found exactly as in \cite{CWW}, see Lemma \ref{fund. sol. lemma 1} in Section \ref{sec Fundamental solution: ingredients}. The bounds $0 < \gamma^{\ast} < n$ hold the existence of $-\gamma$-homogeneous supersolutions $u_{\gamma}$ of
\begin{eqnarray}
\Delta^{\beta,\alpha} u_{\gamma} \geq 1 \quad \text{on}\  \partial \Theta^{\mu \beta}_{1}, \quad \text{for all} \ \alpha_{0} < \alpha < 2/b_{\max},
\end{eqnarray}
see Lemma \ref{barrier lemma lemma 1} in Section \ref{sec Fundamental solution: ingredients}.
\item(Background on the viscosity solutions theory) Comparison principle and stability properties for viscosity solutions in the anisotropic context are provided in \cite{CLU, SL} and references therein. As in \cite{CWW}, the maximum principle and the strong maximum principle for solutions taking non-positive values near of origin are proved in Section \ref{sec Fundamental solution: ingredients}.

\item($L^{q}$-estimates and $C_{\Vert \cdot \Vert_{\beta}}^{\theta}$-regularity for functions in $H^{\alpha, q}_{\beta}\left( \mathbb{R}^{n}\right)$) Schauder estimates and $L_{q}$-estimates for strong solutions in $H^{\alpha, q}_{\beta}\left( \mathbb{R}^{n}\right)$ to the anisotropic Schr\" odinger equation
$$
\Delta^{\beta, \alpha}u (x) + \kappa u = f, \quad \kappa \geq 0, \quad f \in L^{q}\left( \mathbb{R}^{n}\right),
$$
can be found in \cite{LEITAO LP}. Theorem \ref{main theorem} provides $C_{\Vert \cdot \Vert_{\beta}}^{\theta}$-regularity for functions in the anisotropic Bessel spaces.

\end{enumerate}

With steps 1, 2, and 3 at hands, we then use $L^{q}$-estimates to build a sequence $(u_{k})$ of strong solutions to
\begin{eqnarray}
\Delta^{\beta, \alpha}u_{k}(x) + 2^{-k}u_{k}(x) = g_{k}\left(\Vert x \Vert^{}_{\mu \beta}\right) \quad \text{for all} \ x \in \mathbb{R}^{n},
\end{eqnarray}
for a truncation $g_{k}$ appropriated to the intrinsic scaling of $\Delta^{\beta, \alpha}$. Theorem \ref{main theorem} assures pointwise vanishing at infinity for each $u_{k}$ which, via a maximum principle, drive us to the non-negativity of $u_{k}$, see Corollary \ref{max princ 1} and \ref{u goes to zero as x tends to infty lemma}. Then, a barrier argument and Schauder estimates imply that $u_{k}$ converges to a $-\gamma$-homogeneous function $u_{\gamma}$. Finally, stability properties and strong maximum principle guarantee that $u_{\gamma}$ is positive and satisfies
$$
\Delta^{\beta, \alpha} u_{\gamma} = 1, \quad \text{on} \ \partial \Theta^{\mu \beta}_{1},
$$
which is enough to prove the following result:
\begin{theorem}
\label{existence of fundamental solutions}
There exists $0< \alpha_{0} < 2/b_{\max}$ that depends only on $\beta$ and $n$ such that if $\alpha_{0} < \alpha < 2/b_{\max}$ there is a non-constant and $-\gamma^{\ast}$-homogeneous viscosity solution $\Psi_{\beta, \alpha}$ of \eqref{anisotropic harmonic equation} with $\displaystyle\sup_{\partial \Theta_{1}}\Psi_{\beta, \alpha} = 1$ and $0 < \gamma^{\ast}< n$ depending on $\beta$, $n$ and $\alpha_{}$. Moreover, $\Psi_{\beta, \alpha}$ is unique in the following sense: if $u$ is a positive and $-\gamma$-homogeneous viscosity solution of \eqref{anisotropic harmonic equation} then
\begin{eqnarray}
\gamma = \gamma^{\ast} \quad \text{and} \quad u = c_{0}\Psi_{\beta, \alpha},
\end{eqnarray}
for some constant $c_{0}>0$ depending on $\beta$, $n$ and $\alpha_{}$.
\end{theorem}
As in \cite{CWW}, Theorem \ref{existence of fundamental solutions} and the Harnack Inequality for viscous solutions of \eqref{anisotropic harmonic equation}, see Section \ref{sec Fundamental solution: ingredients}, allow us to get a Louville-type Theorem and, as a consequence, we characterize the isolated singularity of a viscous solution $u$ of \eqref{anisotropic harmonic equation}, where $u$ is positive in $\mathbb{R}^{n}$ and locally bounded except at the origin. We also prove $C_{\Vert \cdot \Vert_{\beta}}^{2/b_{\max}}$-regularity for fundamental solutions of \eqref{anisotropic harmonic equation}, see Section \ref{sec Existence of Fundamental solutions and consequences}.

We stress that the condition $c_{\beta} = n$ comes from the theory developed by Lizorkin in \cite{LIZORKIN 1, LIZORKIN 2, LIZORKIN 3}, that is, if we get a version of Theorem \ref{main theorem} for any $c_{\beta} > 0$, our results remain valid. The assumption on the fractional constant $\alpha$ is due to the construction of our barriers and the theory developed in \cite{CLU, SL}. We further mention that, in the context of anisotropic non-local elliptic equations, the usefulness of the barriers built here transcend the purposes of this paper.

The paper is organized as follows. In Section 2 we obtain the tools required to prove Theorems \ref{main theorem} and \ref{existence of fundamental solutions}: decay of oscillation of the average of $u$ in anisotropic balls, $H_{\beta}^{\alpha, q}(\mathbb{R}^{n})$ is continuous embedding in $W_{\beta}^{\alpha, q}(\mathbb{R}^{n})$, and the necessary ingredients on viscosity solution theory to build anisotropic fundamental solutions. Section 3 is devoted to the proof of Theorem \ref{main theorem} and, as a consequence, we show that functions in anisotropic Bessel spaces satisfying pointwise vanishing at infinity when $c_{\beta} = n$. In Section 4 we prove the existence of fundamental solutions at origin of \eqref{anisotropic harmonic equation} and, as a consequence, we get a Louville-type Theorem and characterization of the isolated singularity of a viscosity solution of \eqref{anisotropic harmonic equation}.

\section{Preliminaries:}
In this section we gather the results that we will systematically use along the work. We divide it into two parts. In the first part we address some elementary facts on the topology generated by anisotropic balls and prove the anisotropic versions of the Campanato theory required to access $C^{\gamma}_{\Vert \cdot \Vert_{\beta}}$-regularity for functions in $W_{\beta}^{\alpha, q}(\mathbb{R}^{n})$. We also show that $H_{\beta}^{\alpha, q}(\mathbb{R}^{n})$ is continuously embedding in $W_{\beta}^{\alpha, q}(\mathbb{R}^{n})$ for $c_{\beta}=n$. In the second part we collect and prove the necessary tools to build anisotropic fundamental solutions.
\subsection{Anisotropic topology and elements of the Campanato theory}
Given $r > 0$, $\beta = (b_{1}, \dots, b_{n})\in \mathbb{R}^{n}_{+}$, and $x \in \mathbb{R}^{n}$, we will consider the anisotropic sets
\begin{eqnarray*}
\Theta^{\beta}_{r}\left( x \right) := \left\lbrace y \in \mathbb{R}^{n}: \Vert y - x \Vert_{\beta} < r \right\rbrace \quad \text{and} \quad E^{\beta}_{r}\left( x \right) := \left\lbrace \left( y_{1}, \dots, y_{n}\right) \in \mathbb{R}^{n} : \sum_{i=1}^{n} \frac{\left( y_{i} - x_{i} \right)^{2}}{r^{\frac{4}{b_{i}}}} < 1 \right\rbrace,
\end{eqnarray*}
where $\mathbb{R}^{n}_{+}=\left\lbrace (x_{1}, \dots, x_{n}) \in \mathbb{R}^{n}: x_{i} > 0 \right\rbrace$.
In the case $x=0,$ we will denote by
$$
\Theta^{\beta}_{r}\left( 0 \right) = \Theta^{\beta}_{r} \quad \text{and} \quad E^{\beta}_{r}\left( 0 \right) =  E^{\beta}_{r}.
$$
Moreover, in order not to overload the notation, we will drop the subscript $\beta$ when no confusion can arise. For example,  $\Theta^{}_{r}\left( x \right) = \Theta^{\beta}_{r}\left( x \right),$  $\Vert \cdot \Vert = \Vert \cdot \Vert_{\beta},$ and $c_\beta=c$.  We will also use the classical average notation
$$
\displaystyle \dashint_{E^{\beta}_{r}\left( x \right) } u(z) \, dz = \dfrac{1}{\vert E^{\beta}_{r}\left( x \right)\vert}\displaystyle \int_{E^{\beta}_{r}\left( x \right) } u(z) \, dz,
$$
for all measurable function $u$. We begin with some basic results on the topology generated by anisotropic balls and the equivalence of anisotropic norms on anisotropic spheres. We also show how anisotropic Campanato's semi-norm and the norm of the anisotropic Gagliardo space are related.

\begin{lemma}
\label{elemntary results}
Given $r > 0$ and $x \in \mathbb{R}^{n}$ the following assertions hold:
\begin{enumerate}
\item $E^{}_{r}(x) \subset \Theta_{r\sqrt{n}}(x) \subset E_{rC}(x)$, where $C>0$ is a constant that depends on $\beta$ and $n$.
\item If $\tau_{1}$ is the topology generated by Euclidean balls $B_{r}(z)$ and $\tau_{2}$ is the topology generated by anisotropic balls $E_{r}(z)$ (or $\Theta_{r}(z)$), then $\tau_{1} = \tau_{2}$.
\item Let $\beta_{1} = (b_{1}, \dots, b_{n})\in \mathbb{R}_{+}^{n}$ and $\beta_{2} = (d_{1}, \dots, d_{n}) \in \mathbb{R}_{+}^{n}$. Then there exists a constant $C>0$ that depends on $\beta_{1}$, $\beta_{2}$, and $n$ such that
\begin{eqnarray}
\label{equivalence between norms on anisotropic sphere 1}
C^{-1} \leq \Vert x \Vert^{}_{\beta_{1}} \leq C, \quad \text{for all} \ x \in \partial \Theta^{\beta_{2}}_{1}.
\end{eqnarray}
\item If $u \in W_{\beta}^{\alpha, q}(\mathbb{R}^{n}),$ then the function $x \rightarrow u_{x, r}$ is continuous in $\mathbb{R}^{n}$.
\item If $u \in W_{\beta}^{\alpha, q}(\mathbb{R}^{n})$ then there exists a constant $C>0$ depending only on $\beta$, $\alpha$, and $n$ such that
\begin{eqnarray}
\Vert u \Vert_{L_{\infty}(\mathbb{R}^{n})} \leq C\Vert u \Vert_{W_{\beta}^{\alpha, q}(\mathbb{R}^{n})} \quad \text{and} \quad \left[ u \right]_{q, \alpha q} \leq C\Vert u \Vert_{W_{\beta}^{\alpha, q}(\mathbb{R}^{n})}.
\end{eqnarray}
\end{enumerate}
\end{lemma}

\begin{proof}
For assertions 1 and 2 see Remark 2.9 in \cite{CLU} and Remark 2.3 in \cite{LEITAO}. Furthermore, if $x \in \partial \Theta^{\beta_{2}}_{1}$, we get
\begin{eqnarray}
\vert x_{i} \vert \leq 1 \quad \text{and} \quad \Vert x \Vert_{\beta_{1}}\leq \sqrt{n}.
\end{eqnarray}
On the other hand, if $\Vert x \Vert_{\beta_{2}} = 1$ there exists $1 \leq i_{0} \leq n$ such that $\vert x_{i_{0}} \vert^{d_{i_{0}}} \geq 1/n$. Thus, we find
\begin{eqnarray}
\Vert x \Vert_{\beta_{1}} \geq \vert x_{i_{0}} \vert^{b_{i_{0}}/2} \geq \left( \dfrac{1}{n}\right)^{b_{i_{0}}/2d_{i_{0}}}.
\end{eqnarray}
We now choose $C = \max \left\lbrace n, n^{-b_{i_{0}}/2d_{i_{0}}}\right\rbrace$ to get \eqref{equivalence between norms on anisotropic sphere 1}. To prove assertion 4 note that
\begin{eqnarray}
\label{elementary set relation}
\chi_{A} = \chi_{A \cap B} + \chi_{(A \setminus B)} \quad \text{for all} \ A, B \subset \mathbb{R}^{n}.
\end{eqnarray}
From \eqref{elementary set relation} it follows that
\begin{eqnarray}
\label{elementary lemma 1}
\vert u_{x, r} - u_{x_{0}, r} \vert & = & \nonumber \dfrac{1}{Cr^{c}}\left | \displaystyle \int_{E_{r}(x)}u(z)\, dz - \displaystyle \int_{E_{r}(x_{0})}u(z)\, dz  \right |^{} \\
& = & \dfrac{1}{Cr^{c}}\left | \displaystyle \int_{(E_{r}(x)\setminus E_{r}(x_{0}))}u(z)\, dz - \displaystyle \int_{(E_{r}(x_{0})\setminus E_{r}(x_{}))}u(z)\, dz  \right |^{} \\ \nonumber
& \leq & \dfrac{1}{Cr^{c}} \left( \displaystyle \int_{(E_{r}(x)\setminus E_{r}(x_{0}))} \vert u(z) \vert \, dz + \displaystyle \int_{(E_{r}(x_{0})\setminus E_{r}(x_{}))} \vert u(z) \vert \, dz \right),
\end{eqnarray}
where $\vert E_{r}(z) \vert = C r^{c}$ for all $z \in \mathbb{R}^{n}$. By  H\" older Inequality, we find
\begin{eqnarray}
\label{elementary lemma 2}
\displaystyle \int_{(E_{r}(x)\setminus E_{r}(x_{0}))} \vert u(z) \vert \, dz & \leq & \vert (E_{r}(x)\setminus E_{r}(x_{0}))\vert^{1-\frac{1}{q}} \left( \displaystyle \int_{(E_{r}(x)\setminus E_{r}(x_{0}))} \vert u(z) \vert^{q} \, dz \right)^{\frac{1}{q}}. \\ \nonumber
\end{eqnarray}
Then, combining \eqref{elementary lemma 1} and \eqref{elementary lemma 2} we get
\begin{eqnarray}
\label{elementary lemma 3}
\vert u_{x, r} - u_{x_{0}, r} \vert & \leq & \dfrac{1}{Cr^{\frac{c}{q}}} \left( \Vert u \Vert_{L_{q}(E_{r}(x)\setminus E_{r}(x_{0}))} + \Vert u \Vert_{L_{q}(E_{r}(x_{0})\setminus E_{r}(x_{}))} \right).
\end{eqnarray}
Since $\vert u \vert^{q}$ is integrable and
$$
\lim_{x \rightarrow x_{0}} \vert (E_{r}(x)\setminus E_{r}(x_{0}))\vert = 0,
$$
we use \eqref{elementary lemma 3} to conclude that $u$ is continuous at $x_{0}$. Finally, we estimate
\begin{eqnarray}
\label{}
\vert u(x) - u_{x_{0}, r} \vert^{q} & = & \nonumber \left | u(x) - \displaystyle \dashint_{E_{r}(x_{0})}u(z)\, dz  \right |^{q} \\
& = & \left |  \displaystyle \dashint_{E_{r}(x_{0})} \left( u(x) - u(z) \right) \, dz  \right |^{q} \\ \nonumber
& \leq & \left(  \displaystyle \dashint_{E_{r}(x_{0})} \left| u(x) - u(z) \right | \, dz  \right)^{q}
\end{eqnarray}
and apply H\" older Inequality to obtain
\begin{eqnarray}
\label{}
\displaystyle \dashint_{E_{r}(x_{0})} \left| u(x) - u(z)   \right| \, dz & \leq & \dfrac{\left | E_{r}(x_{0}) \right |^{1 - 1/q }}{\left | E_{r}(x_{0}) \right |} \left(\int_{E_{r}(x_{0})} \left| u(x) - u(z) \right |^{q} \, dz \right )^{\frac{1}{q}} \\ \nonumber
\end{eqnarray}
and deduce that
\begin{eqnarray}
\label{elementary lemma 4}
\vert u(x) - u_{x_{0}, r} \vert^{q} & \leq & \displaystyle \dashint_{E_{r}(x_{0})} \left| u(x) - u(z) \right |^{q} \, dz.
\end{eqnarray}
Integrating \eqref{elementary lemma 4} on $E_{r}(x_{0})$ we get

\begin{eqnarray}
\label{}
\int_{E_{r}(x_{0})}\vert u(x) - u_{x_{0}, r} \vert^{q} \, dx & \leq & \displaystyle \dfrac{1}{\left | E_{r}(x_{0}) \right |}\int_{E_{r}(x_{0})} \int_{E_{r}(x_{0})} \left| u(x) - u(z) \right |^{q} \, dz \,dx \\ \nonumber
& \leq & \displaystyle \dfrac{(C_{1}r)^{c + \alpha q}}{\left | E_{r}(x_{0}) \right |}\int_{E_{r}(x_{0})} \int_{E_{r}(x_{0})}\dfrac{\left| u(x) - u(z) \right |^{q}}{\Vert x - z \Vert^{c + \alpha q}} \, dz \,dx \\ \nonumber
& \leq & \displaystyle C_{2}r^{\alpha q} \Vert u \Vert^q_{W_{\beta}^{\alpha, q}(\mathbb{R}^{n})},
\end{eqnarray}
where we have used assertion 1 to find
$$
x, y \in E_{r}(x_{0}) \subset \Theta_{Cr}.
$$
Furthermore, for all $x \in \mathbb{R}^{n}$, we estimate
\begin{eqnarray}
\label{elementary lemma 5}
\vert u(x) \vert & \leq & \vert u_{x,1} \vert + \vert u_{x,1} - u(x) \vert \\ \nonumber
&\leq & \dfrac{1}{\vert E_{1}(x)\vert^{\frac{1}{q}}} \Vert u \Vert_{L_{q}(\mathbb{R}^{n})} + C_{5} \left[ u\right]_{q, \alpha q}\\ \nonumber
&\leq & C_{6} \Vert u \Vert_{W_{\beta}^{\alpha, q}(\mathbb{R}^{n})}.
\end{eqnarray}
Hence, assertion 5 is proved.
\end{proof}

The next lemma provides the decay of oscillation of the average of $u \in W_{\beta}^{\alpha, q}(\mathbb{R}^{n})$ in anisotropic balls.

\begin{lemma}
\label{decay of average lemma}
Given $1 < q < +\infty$ and $u \in W_{\beta}^{\alpha, q}(\mathbb{R}^{n})$ assume that $c < \alpha \leq c + q$ and $0 < r < R$. Then there exists a constant $C> 0$ depending on $\beta$, $\alpha$, and $n$ such that
\begin{eqnarray}
\label{decay of average lemma 1}
\left | u_{x_{0},R} - u_{x_{0},r} \right | \leq C \left[ u \right]_{q, \alpha q} R^{\theta}, \quad \text{for} \ \theta = \alpha - \dfrac{c}{q},
\end{eqnarray}
for all $x_{0} \in \mathbb{R}^{n}$.
\end{lemma}

\begin{proof}
Given $\sigma_{1}, \sigma_{2} \in \left [r, R \right]$ with $\sigma_{1} < \sigma_{2}$ note that
\begin{eqnarray}
\label{lemma decay oscl 1.1}
\vert u_{x_{0}, \sigma_{2}} - u_{x_{0}, \sigma_{1}} \vert & = & \nonumber \left | u_{x_{0}, \sigma_{2}} - \displaystyle \dashint_{E_{\sigma_{1}}(x_{0})}u(x)\, dx  \right | \\ \nonumber
& = & \left |  \displaystyle \dashint_{E_{\sigma_{1}}(x_{0})} \left( u_{x_{0}, \sigma_{2}} - u(x) \right) \, dx  \right | \\
& \leq & \displaystyle \dashint_{E_{\sigma_{1}}(x_{0})} \left| u(x) - u_{x_{0}, \sigma_{2}} \right| \, dx.
\end{eqnarray}
From H\" older Inequality, it follows that
\begin{eqnarray}
\label{lemma decay oscl 1.2}
\displaystyle \dashint_{E_{\sigma_{1}}(x_{0})} \left| u_{x_{0}, \sigma_{2}} - u(x) \right| \, dx & \leq & \dfrac{\left | E_{\sigma_{1}}(x_{0}) \right |^{1 - 1/q }}{\left | E_{\sigma_{1}}(x_{0}) \right |} \left(\int_{E_{\sigma_{1}}(x_{0})} \left| u(x) - u_{x_{0}, \sigma_{2}} \right |^{q} \, dx \right )^{\frac{1}{q}} \\ \nonumber
& = & \left(\dashint_{E_{\sigma_{1}}(x_{0})} \left| u(x) - u_{x_{0}, \sigma_{2}} \right |^{q} \, dx \right )^{\frac{1}{q}}.
\end{eqnarray}
Moreover, taking into account that $E_{\sigma_{1}}(x_{0}) \subset E_{\sigma_{2}}(x_{0})$, we estimate
\begin{eqnarray}
\label{lemma decay oscl 1.3}
\dashint_{E_{\sigma_{1}}(x_{0})} \left| u(x) - u_{x_{0}, \sigma_{2}} \right |^{q} \, dx & = & \dfrac{\sigma^{\alpha q}_{2}}{\left | E_{\sigma_{1}}(x_{0}) \right |\sigma^{\alpha q}_{2}}\int_{E_{\sigma_{1}}(x_{0})} \left| u(x) - u_{x_{0}, \sigma_{2}} \right |^{q} \, dx \\ \nonumber
& \leq &  \dfrac{\sigma^{c + \alpha q}_{2}}{C\sigma^{c}_{1}}\left( \sigma^{-\alpha q}_{2} \int_{E_{\sigma_{2}}(x_{0})} \left| u(x) - u_{x_{0}, \sigma_{2}} \right |^{q} \, dx \right).
\end{eqnarray}
Combining \eqref{lemma decay oscl 1.1}, \eqref{lemma decay oscl 1.2}, and \eqref{lemma decay oscl 1.3}, we deduce that
\begin{eqnarray}
\label{lemma decay oscl 1}
\vert u_{x_{0}, \sigma_{2}} - u_{x_{0}, \sigma_{1}} \vert & \leq & C^{-1} \left( \sigma^{-\frac{c}{q}}_{1}\sigma^{\alpha}_{2}\right) \left[ u \right]_{q, \alpha q} = C^{-1} \left[ \sigma^{\theta}_{1}\left(\dfrac{\sigma_{2}}{\sigma_{1}}\right)^{\alpha} \right] \left[ u \right]_{q, \alpha q}.
\end{eqnarray}
Putting $\sigma_{1} = 2^{-k}R$ and $\sigma_{2} = 2^{-k+1}R$ in \eqref{lemma decay oscl 1} we get
\begin{eqnarray}
\label{lemma decay oscl 2}
\vert u_{x_{0}, R} - u_{x_{0}, 2^{-k}R} \vert & \leq & R^{\theta}\left(2^{\alpha}C^{-1} \left[ u \right]_{q, \alpha q} \right)\left( \sum_{m=1}^{k}2^{-m \theta} \right) \\ \nonumber
& \leq & C_{2}\left[ u \right]_{q, \alpha q}R^{\theta},
\end{eqnarray}
where $C_{2} = \dfrac{C^{-1}}{2^{\theta} -1}$. Choose $k$ such that $2^{-k}R \leq r < 2^{-k+1}R$. Then, from \eqref{lemma decay oscl 2} we obtain
\begin{eqnarray}
\label{lemma decay oscl 3}
\vert u_{x_{0}, r} - u_{x_{0}, 2^{-k}R} \vert  \leq  C_{2}\left[ u \right]_{q, \alpha q}r^{\theta} \leq C_{2}\left[ u \right]_{q, \alpha q}R^{\theta}.
\end{eqnarray}
Combining \eqref{lemma decay oscl 2} and \eqref{lemma decay oscl 3} the result is proved.
\end{proof}

We now establish the non-trivial connection between spaces $H_{\beta}^{\alpha, q}(\mathbb{R}^{n})$ and $W_{\beta}^{\alpha, q}(\mathbb{R}^{n})$ for $c_{\beta} = n$.
\begin{lemma}
\label{Main. Lemma}
Given $\beta=\left( b_{1},\dots, b_{n} \right) \in \mathbb{R}^{n}_{+}$ and $1 < q < +\infty$, we have:
\begin{eqnarray}
\label{Main. Lemma 1}
L^{\beta}_{q}(\mathbb{R}^{n}) = H_{r_{\beta^{}}}^{\frac{2n}{c_{\beta}}, q}(\mathbb{R}^{n}) \quad \text{and} \quad H_{r_{\beta^{}}}^{\frac{n}{c_{\beta}} \alpha, q}(\mathbb{R}^{n}) \hookrightarrow W_{\beta}^{\alpha, q}(\mathbb{R}^{n}),
\end{eqnarray}
for all $0 < \alpha < 2/b_{\max}$. In particular, if $\beta$ is such that $c_{\beta} = n$ and $2 \leq q < +\infty$, we get
$$
H_{r_{\beta^{}}}^{\alpha, q}(\mathbb{R}^{n}) \hookrightarrow W_{\beta}^{\alpha, q}(\mathbb{R}^{n}).
$$
\end{lemma}
\begin{proof}

Let $\beta^{\ast} = \frac{c_{\beta}}{n}\beta$. By Theorem 1 in \cite{LIZORKIN 2}, we find
\begin{eqnarray}
L^{\beta}_{q}(\mathbb{R}^{n}) = H_{r_{\beta^{}}}^{\alpha^{\ast}, q}(\mathbb{R}^{n}),
\end{eqnarray}
where $\alpha^{\ast}$ and $b^{\ast}_{i}$ satisfy
\begin{eqnarray}
\alpha^{\ast} = \dfrac{n}{\sum_{i=1}^{n}\frac{1}{b_{i}}} = \dfrac{2n}{c_{\beta}} \quad \text{and} \quad  b^{\ast}_{i} = \dfrac{2 b_{i}}{\alpha^{\ast}} = \dfrac{c_{\beta}b_{i}}{n}.
\end{eqnarray}
For the second part, if $\beta(\alpha) = (\alpha b_{1}/ 2, \dots, \alpha b_{n} / 2)$ we get
\begin{eqnarray}
c_{\beta(\alpha)} = \dfrac{2}{\alpha} c_{\beta} \quad \text{and} \quad \frac{c_{\beta(\alpha)}}{n}\beta(\alpha) = \frac{c_{\beta}}{n} \beta = \beta^{\ast}.
\end{eqnarray}
Therefore, from \eqref{Main. Lemma 1}, we obtain
\begin{eqnarray}
L^{\beta(\alpha)}_{q}(\mathbb{R}^{n}) = H_{r_{\beta^{}}}^{\frac{n \alpha}{c_{\beta}}, q}(\mathbb{R}^{n}).
\end{eqnarray}
On the other hand, by Proposition 2.1 in \cite{LEITAO LP} we have
\begin{eqnarray}
W_{\beta}^{\alpha, q}(\mathbb{R}^{n}) = B^{\mu(\alpha, \beta)}_{q},
\end{eqnarray}
where $\mu = (\mu_{1}(\alpha, b_{1}), \dots, \mu_{n}(\alpha, b_{n}))$ is given by $\mu_{i}(\alpha, b_{i}) = \frac{\left[(1 -\nu_{i})q -1 \right]}{q}$ with
\begin{eqnarray}
\label{Main. Lemma 2}
\nu_{i} - 1 = \dfrac{b_{i}}{2}\left[(\nu-1) + \dfrac{1}{q}\left( 1 - \frac{2}{b_{i}}\right)\right] \quad \text{and} \quad (1-\nu)q = q\alpha + 1.
\end{eqnarray}
Moreover, from \eqref{Main. Lemma 2} it follows that
\begin{eqnarray}
\label{Main. Lemma 3}
\left (1- \nu_{i}\right)q & = & \dfrac{b_{i}}{2}\left[(1-\nu)q - \left( 1 - \frac{2}{b_{i}}\right)\right] \\ \nonumber
& = & \dfrac{b_{i}}{2}\left[\alpha q + 1  - \left( 1 - \frac{2}{b_{i}}\right)\right] \\ \nonumber
& = & q \dfrac{\alpha b_{i}}{2} + 1.
\end{eqnarray}
Hence, combining \eqref{Main. Lemma 2} and \eqref{Main. Lemma 3} we find
\begin{eqnarray}
\label{Main. Lemma 4}
W_{\beta}^{\alpha, q}(\mathbb{R}^{n})= B^{\beta(\alpha)}_{q}.
\end{eqnarray}
Finally, from \eqref{Main. Lemma 4} and Theorem 9 in \cite{LIZORKIN 2}, we establish the continuous embedding
\begin{eqnarray}
H_{r_{\beta^{}}}^{\alpha, q}(\mathbb{R}^{n}) \hookrightarrow W_{\beta}^{\alpha, q}(\mathbb{R}^{n}).
\end{eqnarray}
\end{proof}

\subsection{Fundamental solution: ingredients}
\label{sec Fundamental solution: ingredients}
\subsubsection{Viscosity solutions and $L^{p}$-estimates}
In this subsection we collect all the results and concepts from the viscosity solutions theory that we will use to build our anisotropic fundamental solutions. $L^{p}$-estimates are also considered. For more details we refer the reader to \cite{CLU, SL, LEITAO LP, CWW} and references therein.
\begin{definition}
\label{C^{1,1} definition}
A function $\varphi$ is said to be $C^{1,1}$ at the point $x$, and we write $\varphi\in C^{1,1}\left( x \right)$, if there is a vector $v \in \mathbb{R}^{n}$ and numbers $M, \eta_{0} >0$ such that
$$
\vert \varphi \left( x + y\right) - \varphi \left( x \right) -  v \cdot y   \vert \leq M \vert y \vert^{2},
$$
for $\vert  x \vert < \eta_{0}$. We say that a function $\varphi$ is $C^{1,1}$ in a set $\Omega$, and we denote $\varphi \in C^{1,1}\left( \Omega \right)$, if the previous holds at every
point, with a uniform constant $M$.
\end{definition}

Note that for $u \in C^{1,1}\left( x \right) \cap L^{1}\left( \mathbb{R}^{n}, w \right)$ the value of $\Delta^{\beta, \alpha}u(x)$ is well defined, where the anisotropic weight $w: \mathbb{R}^{n} \rightarrow \mathbb{R}$ is given by
\begin{eqnarray}
w(x) = \dfrac{1}{1 + \Vert x \Vert^{c+2\alpha}}.
\end{eqnarray}

Next we introduce the notion of viscosity supersolution (and subsolution) $u$ in a domain $\Omega$, with $C^{2}$ test functions that touch $u$ from above or from below.

\begin{definition}
Let $f$ be a bounded and continuous function in $\mathbb{R}^{n}$. A function $u:\mathbb{R}^{n} \rightarrow \mathbb{R}$, upper (lower) semicontinuous in $\overline{\Omega}$, is said to be a subpersolution (subsolution) to equation $\Delta^{\beta, \alpha}u = f$, and we write $\Delta^{\beta, \alpha}u \geq f$ ($\Delta^{\beta, \alpha}u \leq f$), if whenever the following happen:
\begin{enumerate}
\item $x_{0} \in \Omega$ is any point in $\Omega$;
\medskip
\item $B_{r}\left( x_{0} \right) \subset \Omega$, for some $r>0$;
\medskip
\item $\varphi \in C^{2}\left(\overline{B_{r}\left( x_{0} \right)} \right)$;
\medskip
\item $\varphi \left( x_{0} \right) = u\left( x_{0} \right)$;
\medskip
\item $\varphi \left( y \right) > u\left( y \right)$ ($\varphi \left( y \right) < u\left( y \right)$) for every $y \in B_{r}\left( x_{0} \right) \setminus \left\lbrace x_{0} \right\rbrace$;
\end{enumerate}
then, if we let
$$
v := \left \{
\begin{array}{lll}
\varphi , & \text{ in } & B_{r}\left( x_{0} \right) \\
u & \text{ in } & \mathbb{R}^{n} \setminus  B_{r}\left( x_{0} \right),
\end{array}
\right.
$$
we have $\Delta^{\beta, \alpha}v\left( x_{0} \right) \geq f\left( x_{0} \right)$ ($\Delta^{\beta, \alpha}v\left( x_{0} \right) \leq f\left( x_{0} \right)$).
\end{definition}

Since we use a strategy inspired by Perron's method to reach anisotropic fundamental solutions for $\Delta^{\beta, \alpha}$, it is natural that tools as stability property, comparison principle, Harnack inequality, and maximum principle be required. More precisely, we make use of the following results:

\begin{lemma}[Stability property]
\label{stability prop}
Let $u_{k}$ be a sequence of continuous functions in an open set $\Omega$ and uniformly bounded in $\mathbb{R}^{n}$. Assume that:
\begin{enumerate}
\item $\Delta^{\beta, \alpha}u = f_{k}$ in $\Omega$.
\item $u_{k} \rightarrow u$ locally uniformly in $\Omega$ and $u_{k} \rightarrow u$ a.e. in $\mathbb{R}^{n}$.
\item $f_{k} \rightarrow f$ locally uniformly in $\Omega$.
\end{enumerate}
Then $\Delta^{\beta, \alpha}u = f$ in $\Omega$.
\end{lemma}

\begin{lemma}[Comparison principle]
\label{Comp. princ.}
Let $\Omega$ be a bounded open set and $u$ and $v$ be two bounded functions in $\mathbb{R}^{n}$ such that
\begin{enumerate}
\item $u$ is upper-semicontinuous and $v$ is lower-semicontinuous in $\overline{\Omega}$;
\item $u \leq v$ in $\mathbb{R}^{n} \setminus \Omega$.
\item $\Delta^{\beta, \alpha}u \geq f$ and $\Delta^{\beta, \alpha}v \leq f$ in the viscosity sense in $\Omega$ for a continuous function $f$.
\end{enumerate}
Then $u \leq v$ in $\Omega$.
\end{lemma}

\begin{theorem}[Harnack Inequality]
\label{Harnack Inequality}
Assume that $u$ is a non-negative viscosity solution of $\Delta^{\beta, \alpha}u = 0$ in $\Theta_{2}$. Suppose that $ \alpha_{0} < \alpha < \frac{2}{b_{\max}}$, for some $\alpha_{0} >0$. Then
$$u \leq C u \left( 0 \right) \quad  in \ \ \Theta_{\frac{1}{2}},$$
where $C=C(\beta, n, \alpha_{}) > 0$ is a constant.
\end{theorem}

\begin{lemma}[Maximum principle]
\label{max principle}
Let $\beta \in \mathbb{R}^{n}_{+}$. Assume that $u \in C^{0}_{}\left(\mathbb{R}^{n}\setminus \left\lbrace 0 \right\rbrace \right) \cap L^{1}_{loc}\left(\mathbb{R}^{n}\right)$ and
$$
\displaystyle\limsup_{\Vert x \Vert \rightarrow 0} u(x) \leq 0.
$$
If $u$ satisfies
\begin{eqnarray}
\Delta^{\beta, \alpha} u \leq 0 \quad \text{in} \ \mathbb{R}^{n}\setminus \left\lbrace 0 \right\rbrace,
\end{eqnarray}
then $u \leq 0$ in $\mathbb{R}^{n}\setminus \left\lbrace 0 \right\rbrace$.
\end{lemma}

Let $1 \leq q < + \infty$ and $0 < \theta < 1$. Since
\begin{eqnarray}
\label{u goes to zero as x tends to infty preliminaries}
\lim_{\Vert x \Vert \rightarrow +\infty } u(x) = 0,
\end{eqnarray}
for every $u \in C^{\theta}_{\Vert \cdot \Vert}\left(\mathbb{R}^{n}\right) \cap L^{q}\left(\mathbb{R}^{n}\right)$, see Corollary \ref{u goes to zero as x tends to infty lemma} in Section 3 for more details, a direct consequence of Lemma \ref{max principle} is a maximum principle for solutions to the anisotropic non-local Schr\"odinger equation.
\begin{corollary}
\label{max princ 1}
Let $\mu >0$ and $\beta \in \mathbb{R}^{n}_{+}$. Assume that $u \in C^{\theta}_{\Vert \cdot \Vert_{\mu \beta}}\left(\mathbb{R}^{n}\right) \cap L^{q}\left(\mathbb{R}^{n}\right)$ for some $0 < \theta < 1$ and $1 \leq q < +\infty$. If $u$ satisfies
\begin{eqnarray}
\Delta^{\beta, \alpha} u + \kappa u \geq 0 \quad \text{in} \ \mathbb{R}^{n},
\end{eqnarray}
for some $k \geq 0$, then $u \geq 0$ in $\mathbb{R}^{n}$.
\end{corollary}
\begin{proof}
We consider two cases:
\begin{enumerate}
\item Case 1: $u(0) \geq  0$. \\
Let us suppose, for the purpose of contradiction, that there exists a $x_{1} \in \mathbb{R}^{n} \setminus \left\lbrace 0 \right\rbrace$ such that $u(x_{1}) < 0$. From \eqref{u goes to zero as x tends to infty preliminaries} we can choose $R > 1$ such that
\begin{eqnarray}
\label{max princ 1 1}
u(x_{1}) - u(x) < 0, \quad \text{for all} \ x \in \mathbb{R}^{n} \setminus \Theta^{\mu \beta}_{R}.
\end{eqnarray}
We now use the continuity and non-negativity of $u$ at $0$ to get $0 < r < \Vert x_{1} \Vert_{\mu \beta }$ such that
\begin{eqnarray}
\label{max princ 1 2}
u(x_{1}) - u(x) < 0, \quad \text{for all} \ x \in \Theta^{\mu \beta}_{r}.
\end{eqnarray}
Since $x_{1} \in \Theta^{\mu \beta}_{R} \setminus \Theta^{\mu \beta}_{r}$, we find
\begin{eqnarray}
\label{max princ 1 3}
u(x_{2}) = \inf_{\Theta^{\mu \beta}_{R} \setminus \Theta^{\mu \beta}_{r}} u \leq u(x_{1}) < 0 \quad .
\end{eqnarray}
Hence
$$
\Delta^{\beta, \alpha} u(x_{2}) + \kappa u(x_{2}) < 0.
$$
\item Case 2: $u(0) <0 $. \\
Let $R > 1$ such that
\begin{eqnarray}
\label{max princ 1 4}
u(0) - u(x) < 0, \quad \text{for all} \ x \in \mathbb{R}^{n} \setminus \Theta^{\mu \beta}_{R}.
\end{eqnarray}
Notice that
$$u(x_{1}) = \inf_{\Theta^{\mu \beta}_{R}} u \leq u(0) < 0.$$
Then, from \eqref{max princ 1 4} it follows that
$$
u(x_{1}) - u(x) < 0, \quad \text{for all} \ x \in \mathbb{R}^{n} \setminus \Theta^{\mu \beta}_{R}
$$
Thus, we find
$$
\Delta^{\beta, \alpha} u(x_{1}) + \kappa u(x_{1}) < 0.
$$
\end{enumerate}
\end{proof}

\begin{lemma}[Strong maximum principle]
\label{strong max principle}
Let $\beta \in \mathbb{R}^{n}_{+}$. Assume that $u \in C^{0}_{}\left(\mathbb{R}^{n}\setminus \left\lbrace 0 \right\rbrace \right) \cap L^{1}_{loc}\left(\mathbb{R}^{n}\right)$ and $u \leq 0$ in $\mathbb{R}^{n}\setminus \left\lbrace 0 \right\rbrace$.
If $u$ satisfies
\begin{eqnarray}
\Delta^{\beta, \alpha} u \leq 0 \quad \text{in} \ \mathbb{R}^{n}\setminus \left\lbrace 0 \right\rbrace,
\end{eqnarray}
and there exists $x \in \mathbb{R}^{n}\setminus \left\lbrace 0 \right\rbrace $ such that $u(x) = 0$ then $u = 0$ in $\mathbb{R}^{n}\setminus \left\lbrace 0 \right\rbrace$.
\end{lemma}

Let $\beta \in \mathbb{R}^{n}_{+}$. In next result $\mu = \mu_{\beta}$ is a natural number such that $\Vert x \Vert_{\mu\beta}$ is a $C^{2}$-function in $\mathbb{R}^{n} \setminus \left\lbrace 0 \right\rbrace$. We now present Schauder's estimates and $L^{q}$-estimates for strong solutions to the anisotropic non-local Schr\" odinger equation that we use in our analysis.

\begin{lemma}
\label{C-regularity in general}
Let $\beta=(b_{1}, \dots, b_{n}) \in \mathbb{R}^{n}_{+}$, $\kappa \geq 0$, $0 < \alpha < \frac{2}{b_{\max}}$, $0 < r < R < 1$ and $f \in L^{\infty}(\Theta^{\mu \beta}_{1})$. If $u \in C^{2}_{0}(\Theta^{\mu \beta}_{1}) \cap L_{1} (\mathbb{R}^{n}, w)$ is a solution of
$$
\Delta^{\beta, \alpha} u + \kappa u = f,  \quad \text{in} \ \Theta^{\mu \beta}_{R},
$$
then we have
\begin{eqnarray*}
\label{C-reg. estimate in general}
\left[ u \right]_{C_{\Vert \cdot \Vert_{\mu \beta}}^{\theta}(\Theta^{\mu \beta}_{r})} \leq C \left( (R - r)^{-\theta} \Vert u \Vert_{L^{\infty}(\Theta^{\mu \beta}_{R})} + (R - r)^{-\frac{c_{\beta}}{\mu}-\theta} \Vert u \Vert_{ L^{1} (\mathbb{R}^{n}, w)} + (R - r)^{2\frac{\alpha}{\mu} - \theta} \Vert f \Vert_{L^{\infty}(\Theta^{\mu \beta}_{R})} \right),
\end{eqnarray*}
for any $\theta \in \left(0, \min \left \lbrace 2/\mu b_{\max}, 2\alpha/\mu \right \rbrace \right)$, where $C=C(\beta, \mu, n) >0$ is a constant.
\end{lemma}

\begin{remark}
\label{C regilarity original norm remark}
If $\mu>0$ is as in previous Lemma we have $\mu b_{i} \geq b_{i}$. We then deduce that
\begin{eqnarray}
\Theta^{\beta}_{r} \subset \Theta^{\mu \beta}_{r} \quad \text{and} \quad \Vert x - y \Vert_{\mu\beta} \leq \Vert x - y \Vert^{\mu}_{\beta}, \quad \text{for all} \ x,y \in \Theta^{\beta}_{r},
\end{eqnarray}
for every $0 < r < 1$. Hence, we can drop $\mu$ in previous Lemma. More precisely, we get
$$
u \in C^{\theta}_{\Vert \cdot \Vert_{\beta}}\left( \partial \Theta^{\beta}_{1/2} \right) \quad \text{if} \ 0 < \theta < \left(0, \min \left \lbrace 2/b_{\max}, 2\alpha \right \rbrace \right).
$$
\end{remark}

\begin{theorem}[$L^{q}$ estimates]
\label{Lq estimates}
Let $1 < q < +\infty$, $\beta = (b_{1}, \dots, b_{n}) \in \mathbb{R}^{n}_{+}$, and $0 < \alpha < \frac{2}{b_{\max}}$. If $\kappa > 0$, there exists a unique strong solution $u \in H_{\beta}^{\alpha,q}(\mathbb{R}^{n})$ such that
\begin{eqnarray}
\Delta^{\beta, \alpha}u + \kappa u = f \quad \text{in} \ \mathbb{R}^{n},
\end{eqnarray}
and
\begin{eqnarray}
\Vert u \Vert_{H_{\beta}^{\alpha, q}(\mathbb{R}^{n})} \leq C \Vert f \Vert_{L^{q}\left( \mathbb{R}^{n}\right)},
\end{eqnarray}
where $C > 0$ is a constant depending on $\beta$, $\alpha$, $q$, and $n$.
\end{theorem}

\subsubsection{$-\gamma$-homogeneous solutions}
Given $\gamma \in \left(0, +\infty \right)$ we consider the set of the continuous functions in $\mathbb{R}^{n}\setminus \left\lbrace 0 \right\rbrace$ that are non-negative and $-\gamma$-homogeneous:
\begin{eqnarray}
\mathcal{A}_{\gamma} = \left\lbrace u \in C^{0}_{}\left(\mathbb{R}^{n}\setminus \left\lbrace 0 \right\rbrace \right): u \geq 0, r^{\gamma} u(T_{\beta, r} x) = u(x) \right\rbrace.
\end{eqnarray}
We also denote by
\begin{eqnarray}
\mathcal{A}^{+}_{\gamma} = \left\lbrace u \in \mathcal{A}_{\gamma}: u > 0 \right\rbrace.
\end{eqnarray}
The scaling exponent of $(-\Delta)^{\beta, \alpha}$ is defined by $\gamma_{\beta, \alpha}^{\ast} = \sup \mathcal{I}_{\beta\alpha}$ where
\begin{eqnarray*}
\mathcal{I}_{\beta\alpha} = \left\lbrace \gamma \in \left(0, +\infty \right) : \text{there exists} \ u \in \mathcal{A}^{+}_{\gamma} \ \text{such that} \ \Delta^{\beta, \alpha}u \geq 0 \ \text{in} \ \mathbb{R}^{n}\setminus \left\lbrace 0 \right\rbrace  \right\rbrace.
\end{eqnarray*}

The first step towards construction of fundamental solutions for $\Delta^{\beta, \alpha}$ is to prove that $\mathcal{I}_{\beta\alpha}$ is not empty. Via a standard scaling argument, see Lemma \ref{scaling lemma on the anisotropic sphere} below, it suffices to find a $-\gamma$-homogeneous function which is a supersolution on $\partial \Theta^{\mu \beta}_{1}$ for some suitable $\mu>0$. Therefore, it is essential to understand how the norm $\Vert \cdot \Vert_{\mu \beta}$ and the scaling anisotropic $T_{\beta, r}$ are related for any $r>0$. This is the content of the next result.

\begin{lemma}[Scaling properties]
\label{scaling properties lemma}
Given $\beta=(b_{1}, \dots, b_{n}) \in \mathbb{R}^{n}_{+}$ and $r>0$, let $T_{\beta,r}: \mathbb{R}^{n} \rightarrow \mathbb{R}^{n}$ be the linear map defined by
\begin{eqnarray}
T_{\beta,r}e_{i} = r^{2/b_{i}}e_{i},
\end{eqnarray}
where $e_{i}$ is the $i$-th canonical vector. If $\mu > 0$ and $x \neq 0$, we have
\begin{eqnarray}
T^{-1}_{\beta,r} = T^{}_{\beta,r^{-1}} \quad \text{and} \quad r^{2\mu}\Vert T^{-1}_{\beta,r} x \Vert^{2}_{\mu \beta} = \Vert x \Vert^{2}_{\mu \beta}.
\end{eqnarray}
In particular:
\begin{enumerate}
\item If $r = \Vert x \Vert_{\mu \beta}^{1/\mu}$ we get
 \begin{eqnarray}
\label{anisotropic inversion}
T^{-1}_{\beta,r}\left( \mathbb{R}^{n} \setminus \left\lbrace 0 \right\rbrace \right) \subset \partial \Theta^{\mu \beta}_{1}.
\end{eqnarray}
\item If $0 < r < 1$ and $0 < \alpha < 2/ b_{\max}$, we find
\begin{eqnarray}
\label{equivalence between norms on anisotropic sphere 2}
\int_{\Theta^{\mu \beta}_{r}} \dfrac{\vert y \vert^{2}}{\Vert y \Vert_{\beta}^{c_{\beta}+2\alpha}} dy < +\infty \quad \text{and} \quad \int_{\mathbb{R}^{n} \setminus \Theta^{\mu \beta}_{r}} \dfrac{1}{\Vert y \Vert_{\beta_{}}^{c_{\beta_{}}+2\alpha}} dy < +\infty.
\end{eqnarray}
\end{enumerate}
\end{lemma}
\begin{proof}
It is immediate that
$$
T^{-1}_{\beta,r} = T^{}_{\beta,r^{-1}}.
$$
Furthermore, if $x \neq 0$, we have
\begin{eqnarray}
\label{scaling of dilatations 1}
\Vert T^{-1}_{\beta,r} \Vert^{2}_{\mu \beta} = \sum_{i=1}^{n} \left| r^{-2/b_{i}} x_{i}\right |^{\mu b_{i}} = \dfrac{\Vert x \Vert^{2}_{\mu \beta}}{r^{2\mu}}.
\end{eqnarray}
Hence, if $r = \Vert x \Vert_{\mu \beta}^{1/\mu}$ in \eqref{scaling of dilatations 1} we reach \eqref{anisotropic inversion}. To prove \eqref{equivalence between norms on anisotropic sphere 2} we use anisotropic polar coordinates. Indeed, we have
\begin{eqnarray}
\label{scaling properties lemma 1}
\int_{\Theta^{\mu \beta}_{r}} \dfrac{\vert y \vert^{2}}{\Vert y \Vert^{c_{\beta_{}}+2\alpha}} dy & = &  \nonumber \int_{0}^{r}\int_{\partial \Theta^{\mu \beta}_{1}} \dfrac{\vert T_{\mu \beta_{},t}y \vert^{2}}{\Vert T_{\mu \beta_{},t}y \Vert_{\beta}^{c_{\beta_{}}+2\alpha}} t^{\frac{1}{\mu}c_{\beta} - 1} dy \, dt \\ \nonumber
& = & \nonumber \int_{0}^{r}\int_{\partial \Theta^{\mu \beta}_{1}} \dfrac{t^{\frac{1}{\mu}c_{\beta}-1}\vert T_{\mu \beta_{},t}y \vert^{2}}{t^{\frac{c_{\beta} + 2\alpha}{\mu}}\Vert y \Vert_{\beta}^{c_{\beta_{}}+2\alpha}} dy \, dt \\
& \leq & C_{1} \int_{0}^{r}\int_{\partial \Theta^{\mu \beta}_{1}}t^{-\frac{2\alpha}{\mu} -1}\vert T_{\mu \beta_{},t}y \vert^{2}dy \, dt,
\end{eqnarray}
where we have used the assertion 3 of Lemma \ref{elemntary results} to obtain the last inequality. From \eqref{scaling properties lemma 1} and
\begin{eqnarray}
\vert T_{\mu \beta_{},t}y \vert^{2} = \sum_{i=1}^{n} \vert t^{4/\mu b_{i}}y_{i}^{2} \vert \leq  t^{\frac{4}{\mu b_{\max}}} \vert y \vert^{2}, \quad \text{for all} \ (t, y) \in \left[0,1 \right] \times    \partial \Theta^{\mu \beta}_{1},
\end{eqnarray}
we find
\begin{eqnarray}
\int_{\Theta^{\mu \beta}_{r}} \dfrac{\vert y \vert^{2}}{\Vert y \Vert^{c_{\beta_{}}+2\alpha}} dy & \leq & C_{1} \left(\int_{0}^{r} t^{\frac{2}{\mu}\left( \frac{2}{b_{\max}} - \alpha \right)-1} dt \right)\left[ \sup_{1 \leq i \leq n} \left(\int_{\partial \Theta^{\mu \beta}_{1}} y^{2}_{i}\, dy \right)\right] \\ \nonumber
& = & C_{2},
\end{eqnarray}
where $C_{2}>0$ is a constant that depends on $r$, $\mu$, $\beta$, $\alpha$, and $n$. Analogously, we deduce that
\begin{eqnarray}
\int_{\mathbb{R}^{n} \setminus \Theta^{\mu \beta}_{r}} \dfrac{1}{\Vert y \Vert^{c_{\beta_{}}+2\alpha}} dy & = & \nonumber \int_{r}^{+\infty}\int_{\partial \Theta^{\mu \beta}_{1}} \dfrac{t^{\frac{1}{\mu}c_{\beta}-1}}{t^{\frac{c_{\beta} + 2\alpha}{\mu}}\Vert y \Vert_{\beta}^{c_{\beta_{}}+2\alpha}} dy \, dt \\ \nonumber
& \leq & C_{3} \left(\int_{r}^{+\infty} t^{-\frac{2\alpha}{\mu}-1} dt \right),
\end{eqnarray}
for some constant $C_{3}>0$ that depends on $r$, $\mu$, $\beta$, $\alpha$, and $n$. The result is proved.
\end{proof}

Next result reveals the relation between the intrinsic scaling of $\Delta^{\beta, \alpha}$ and $-\gamma$-homogeneous functions. Its proof is a direct consequence of Lemma \ref{scaling properties lemma}, see also Proposition 1 in \cite{CWW}.
\begin{lemma}
\label{scaling lemma on the anisotropic sphere}
Let $0 < \gamma < +\infty$. If $u \in \mathcal{A}_{\gamma}$ is such that
$$
\Delta^{\beta, \alpha}u \geq 1 \quad \text{on} \ \partial \Theta^{\mu \beta}_{1},
$$
for some $\mu >0$. Then we get
\begin{eqnarray}
\Delta^{\beta, \alpha}u(x) \geq \dfrac{1}{\Vert x \Vert_{\mu \beta}^{\left( \frac{\gamma}{\mu} + 2 \frac{\alpha}{\mu} \right)}}\quad \text{in} \ \mathbb{R}^{n} \setminus \left\lbrace 0 \right\rbrace.
\end{eqnarray}
\end{lemma}

With Lemmas \ref{scaling properties lemma} and \ref{scaling lemma on the anisotropic sphere} at hands we show that there exists a constant $0 < \alpha_{0} < 2/b_{\max}$ depending on $\beta$ and $n$ such that $\mathcal{I}_{\beta\alpha}$ is not empty if $\alpha_{0} < \alpha_{} < 2/b_{\max}$ and $\gamma$ is sufficiently small. Next, we list some technical and geometric facts that we use to build anisotropic barriers. We begin by highlighting some features of a special class of cones in $\mathbb{R}^{n}$. Given $\beta=(b_{1}, \dots, b_{n}) \in \mathbb{R}^{n}_{+}$, $0 < \delta < 1$ and $x \in \partial \Theta^{\beta}_{1}$ we define the anisotropic cone (symmetric set) by
\begin{eqnarray}
\Omega^{\beta, \delta}_{x} := \left \lbrace y \in \mathbb{R}^{n}: \ \vert \langle x, y \rangle_{x}\vert \leq (1 - \delta) \Vert x \Vert_{x} \Vert y \Vert_{x}\right \rbrace,
\end{eqnarray}
where the inner product $\langle y, z \rangle_{x} : = \displaystyle \sum_{i=1}^{n}d_{i}(x) y_{i}z_{i}$ and the natural norm $\Vert y \Vert_{x} = \langle y, y \rangle_{x}^{1/2}$ are defined by
$$
d_{i}(x) :=\left \{
\begin{array}{ll}
\vert x_{i}\vert^{b_{i} - 2}, \ \text{ if } \ x_{i} \neq 0, \\
\\
1, \text{ if } \  x_{i}= 0.
\end{array}
\right.
$$
If $\beta_{0} = (2, \dots, 2)$ we consider $G_{\beta, 1}: \partial \Theta^{\beta}_{1} \rightarrow \partial \Theta^{\beta_{0}}_{1}$ defined by
\begin{eqnarray*}
G_{\beta, 1}(y)= \left(d^{1/2}_{1}(y) y_{1} , \dots, d^{1/2}_{1}(y) y_{n} \right).
\end{eqnarray*}
The map $G_{\beta, 1}$ allows us to establish a fine relation between our anisotropic cones and classical cones, that is,
\begin{eqnarray}
\label{fine geometric asymptotic condition on alpha 0.1}
\Omega^{\beta_{}, \delta}_{x} = \mathcal{G}^{}_{\beta, x}\left(\Omega^{\beta_{0}, \delta}_{G_{\beta, 1}(x)} \right), \quad \text{for all} \ \beta_{} \in \mathbb{R}^{n}_{+} \ \text{and} \ x \in \partial \Theta^{\beta}_{1},
\end{eqnarray}
where $\mathcal{G}^{}_{\beta, x}(y) = \left( d^{-1/2}_{1}(x) y_{1}, \dots, d^{-1/2}_{n}(x) y_{n} \right)$ for all $y \in \mathbb{R}^{n}$. Bearing in mind the change of variables theorem, \eqref{fine geometric asymptotic condition on alpha 0.1} is interesting because help us to estimate integrals on anisotropic cones via integrals on classical cones. For example, given $\mu>0$ and $x \in \partial \Theta^{\mu \beta}_{1}$ we have
\begin{eqnarray}
\label{integral on anisotropic cones via integral on classcal cones}
\int_{\Omega^{\mu \beta, \delta_{}}_{x}\cap \Theta^{\beta}_{r_{1}}} f(y) dy & = & \int_{\Omega^{\beta_{0}, \delta_{}}_{\tilde{x}}\cap \mathcal{G}^{-1}_{\mu \beta,x}(\Theta^{\beta}_{r_{1}})} f(\mathcal{G}^{}_{\mu \beta,x}(y))\vert \det \mathcal{G}^{}_{\mu \beta, x}(y) \vert dy \\ \nonumber
& = & \int_{\Theta^{\beta}_{r_{1}}} f(y) dy - \int_{\Theta^{\beta}_{r_{1}}} f(y)\chi_{\left((\mathbb{R}^{n} \setminus \Omega^{\beta_{0}, \delta_{}}_{\tilde{x}}) \cap \mathcal{G}^{-1}_{\mu \beta,x}(\Theta^{\beta}_{r_{1}})\right)}(y) dy \\ \nonumber
& \geq & \int_{\Theta^{\beta}_{r_{1}}} f(y) dy - \int_{\Theta^{\beta}_{r_{1}}} f(y)\chi_{ \left( \mathbb{R}^{n} \setminus \Omega^{\beta_{0}, \delta_{}}_{\tilde{x}}\right)}(y) dy,
\end{eqnarray}
for every $0 < r_{1} < 1$ and $f \in L^{1}_{loc}\left( \mathbb{R}^{n} \right)$, where we have denoted by $\tilde{x} = G_{\mu\beta, 1}(x)$.
Another crucial step to find $\gamma \in \mathcal{I}_{\beta\alpha}$ is to estimate from below (universally) the integral of the quadratics functions $y^{2}_{i}$ for $b_{i}=b_{\max}$ on $\partial\Theta^{\beta_{}}_{1}$. In fact, let
$$
\mathcal{S}_{0}^{\beta_{},n} = \left\lbrace y \in \partial\Theta^{\beta_{0}}_{1}: \sum_{b_{i}= b_{\max}}^{} y^{2}_{i} \geq \frac{1}{n} \right\rbrace \quad \text{and} \quad \mathcal{S}^{\beta,n}_{} = G_{\beta,2}^{}(\mathcal{S}^{\beta_{0},n}_{0}),
$$
where $G_{\beta,2}: \partial\Theta^{\beta_{0}}_{1} \rightarrow \partial\Theta^{\beta_{}}_{1}$ is natural map given by
\begin{eqnarray*}
G_{\beta,2}(y)= \left(y_{1}^{2/b_{1}}, \dots, y_{n}^{2/b_{n}} \right).
\end{eqnarray*}
Clearly,
\begin{eqnarray*}
\inf_{\mathcal{S}^{\beta,n}_{}} \sum_{b_{i}= b_{\max}}^{} y^{2}_{i} \geq c_{0}
\end{eqnarray*}
for some $c_{0}>0$ depending on $\beta$ and $n$, and therefore we deduce that
\begin{eqnarray}
\label{bound from below on S_beta_delta_zero}
\int_{\partial\Theta^{\beta}_{1}} \sum_{b_{i}= b_{\max}}^{} y^{2}_{i} dy \geq c_{0}\vert \mathcal{S}^{\beta,n} \vert.
\end{eqnarray}
We now choose $1/2 < \delta_{0} < 1$ (depending on $\beta$ and $n$) such that
\begin{eqnarray}
\label{geometric asymptotic condition on alpha}
\vert \left(\mathbb{R}^{n} \setminus \Omega^{\beta_{0},\delta_{0}}_{x} \right) \cap \partial\Theta^{\beta_{}}_{1} \vert \leq \dfrac{c_{0}\vert \mathcal{S}^{\beta,n} \vert }{2n}, \quad \text{for all} \ x \in \partial\Theta^{\beta_{0}}_{1}.
\end{eqnarray}
For $\gamma > 0$ and $r>0$ let us introduce the following notation
\begin{eqnarray}
\label{tilde notation}
\tilde{\gamma} = \dfrac{\gamma}{\mu_{\beta}},\quad \tilde{\beta} = \mu_{\beta} \beta, \quad  \text{and} \quad \tilde{\Theta}_{r} = \Theta^{\tilde{\beta}}_{r},
\end{eqnarray}
where $\mu_{\beta}$ is a natural number such that $\Vert x \Vert_{\tilde{\beta}}$ is a $C^{2}$-function in $\mathbb{R}^{n} \setminus \left\lbrace 0 \right\rbrace$ and
\begin{eqnarray}
\label{condition on mu beta geometric asymptotic condition on alpha}
 \left( 1 - \delta_{0} \right)^{2} - 1 + \dfrac{1}{\tilde{b}_{\max}} < 0.
\end{eqnarray}
Furthermore, given $x_{0} \in \partial \tilde{\Theta}_{1}$ and $c_{1} \in \left(0, 1 \right)$ let
\begin{eqnarray}
\mathcal{B}_{x_{0}} := L^{-1}_{\tilde{\beta}}\left( \Omega^{\tilde{\beta}, \delta_{0}}_{x} \right) \cap \tilde{\Theta}_{c_{1}},
\end{eqnarray}
where $L_{\tilde{\beta}}: \mathbb{R}^{n} \rightarrow \mathbb{R}^{n}$ is the linear map defined by $L_{\tilde{\beta}} e_{i} =  \tilde{b}_{i} e_{i}$. Let us to make a technical observation on the set $\mathcal{B}_{x_{0}}$ which will be useful in our construction of barriers. From the relation
\begin{eqnarray}
\Theta_{r_{1}} \subset \tilde{\Theta}_{r_{1}} \subset L^{}_{\tilde{\beta}}(\tilde{\Theta}_{c_{1}}), \quad \text{for} \ r_{1} = \tilde{b}^{\tilde{b}_{\min}/2}_{\min}c_{1} < 1,
\end{eqnarray}
and change of variables it follows that
\begin{eqnarray}
\label{technical barrier integral estimate}
\int_{\mathcal{B}_{x}} \dfrac{y^{2}_{i}}{\Vert y \Vert^{c+2\alpha}} dy & \geq & m_{0}\int_{\Omega^{\tilde{\beta}, \delta_{0}}_{x}\cap \Theta_{r_{1}}} \dfrac{y^{2}_{i_{}}}{\Vert y \Vert^{c+2\alpha}} dy,
\end{eqnarray}
for some $m_{0} > 0$ depending on $\tilde{\beta}$ and $n$. The next technical lemma is also fundamental to get our barrier.
\begin{lemma}
\label{technical lemma barrier}
Let $i_{0} \in \left\lbrace 1, \dots, n \right\rbrace$. There exist constants $a_{ij}\in \mathbb{R}$ depending on $n$ and $i_{0}$ such that $a_{ii} > 0$ and
\begin{eqnarray}
\label{technical lemma barrier inequality}
y^{2}_{i_{0}} = \sum_{i=1}^{n} a_{ii}y_{i}^{2} + \sum_{i\neq j}^{} a_{ij} y_{i}y_{j}.
\end{eqnarray}
\end{lemma}
\begin{proof}
Notice that
\begin{equation}
y^{2}_{i_{0}} = \langle y, e_{i_{0}}\rangle^{2} = \langle T_{i_{0}}y, T_{i_{0}} e_{i_{0}}\rangle^{2},
\end{equation}
where $T_{i_{0}}: \mathbb{R}^{n} \rightarrow \mathbb{R}^{n}$ is an orthogonal linear map such that
$T_{i_{0}}e_{i_{0}} = \left(1/\sqrt{n}, \dots, 1/\sqrt{n} \right)$.
\end{proof}

Hereafter, the symbol $C_{0}$ will denote a positive constant depending on $\beta$ and $n$ for which $\Vert \cdot \Vert$ satisfies the quasi-triangular inequality condition, that is,
\begin{eqnarray}
\Vert x + y \Vert \leq C_{0}\left(\Vert x \Vert + \Vert y \Vert \right), \quad \text{for all} \ x,y \in \mathbb{R}^{n}.
\end{eqnarray}

\begin{lemma}[Barrier function]
\label{Bar. func.}
Given $\gamma >0$ and $\beta=\left(b_{1}, \dots, b_{n} \right) \in \mathbb{R}^{n}_{+}$, let $u_{\gamma}: \mathbb{R}^{n} \rightarrow \mathbb{R}$ defined by $u_{\gamma}(x) = \Vert x \Vert_{\tilde{\beta}}^{-\tilde{\gamma}}$. There exist constants $0 < \gamma < 1/2$ and $0 < \alpha_{0} < 2/b_{\max}$ depending on $\beta$ and $n$ such that
\begin{eqnarray}
\label{barrier function inequality}
\Delta^{\beta, \alpha} u_{\gamma} \geq c_{0} \quad \text{on} \ \partial \tilde{\Theta}_{1}, \quad \text{if} \ \alpha_{0} < \alpha < 2/b_{\max},
\end{eqnarray}
for some $c_{0} > 0$ depending on $\beta$ and $n$. In particular, $u_{\gamma} \in \mathcal{A}^{+}_{\gamma}$.
\end{lemma}
\begin{proof}
Let $x_{0} \in \partial \tilde{\Theta}_{1}$ and $c_{1} \in \left(0,1 \right)$. Since $\mathcal{B}_{x_{0}}$ is symmetric, the strategy to prove \eqref{barrier function inequality} is to write
\begin{eqnarray}
\label{barrier function inequality 1}
\dfrac{\Delta^{\beta, \alpha} u_{\gamma}(x_{0})}{C_{\beta, \alpha}}& = & \nonumber -\dfrac{1}{2}\displaystyle \int_{\mathcal{B}_{x_{0}}} \dfrac{u_{\gamma}\left( x_{0} + y \right) + u_{\gamma}\left( x_{0} - y \right) - 2}{\Vert y \Vert^{c_{} + 2\alpha}} dy + \displaystyle \int_{\mathbb{R}^{n} \setminus \mathcal{B}_{x_{0}}} \dfrac{1 - u_{\gamma}\left( x_{0} + y\right)}{\Vert y \Vert^{c_{} + 2\alpha}} dy \\
& = & I_{1} + I_{2},
\end{eqnarray}
and show there exist constants $c_{0} > 0$ and $\alpha_{0}>0$ depending $\beta$ and $n$ such that $I_{1} \geq c_{0}$ and $ I_{2}  \geq - c_{0}/2$ for some constant $c_{0}=c_{0}(\beta, n)>0$, if $\alpha > \alpha_{0}$ and $\gamma$ is sufficiently small. We begin by estimating $I_{1}$ by bellow. Indeed, from Lemma 2.12 in \cite{LEITAO LP}, $u_{\gamma}$ is a function in $C^{2}(\mathbb{R}^{n}\setminus \left\lbrace 0 \right\rbrace)$ such that
\begin{eqnarray}
\partial^{2}_{ij}u_{\gamma}(x):= \left \{
\begin{array}{ll}
\dfrac{\tilde{\gamma}(\tilde{\gamma}+2)}{4}\Vert x \Vert_{\tilde{\beta}}^{\tilde{\gamma} - 4} (\tilde{b}_{i}\vert x_{i}\vert^{\tilde{b}_{i}-2} x_{i}) (\tilde{b}_{j}\vert x_{j}\vert^{\tilde{b}_{j}-2} x_{j}), \ \text{ if } \ i \neq j, \\
\\
\dfrac{\tilde{\gamma}}{4}\Vert x \Vert_{\tilde{\beta}}^{-\tilde{\gamma} - 4} \left \lbrace \tilde{b}^{2}_{i}\vert x_{i} \vert^{2(\tilde{b}_{i} -1)} (\tilde{\gamma} + 2) - \Vert x \Vert^{2}_{\tilde{\beta}} \left [ 2\tilde{b}_{i}(\tilde{b}_{i} - 1) \vert x_{i} \vert^{\tilde{b}_{i}-2}\right] \right \rbrace, \\
 \text{ if } \  i = j,
\end{array}
\right.
\end{eqnarray}
for all $x \neq 0$. From item 3 of Lemma 2.12 in \cite{LEITAO LP}, given $y \in \mathcal{B}_{x_{0}}$, there exists a constant $C>0$ depending on $\beta$ and $n$ such that $G_{\gamma}(t) = u_{\gamma}(x_{0} + ty)$ satisfies
\begin{eqnarray}
\label{Bar. func. second derivat est.}
G_{\gamma}''(t) \leq \tilde{\gamma} \Vert x + t y \Vert_{\tilde{\beta}}^{\tilde{\gamma} - 4} \left \lbrace \displaystyle \sum_{i=1}^{n} \vert x_{i} \vert^{\tilde{b}_{i} -2}\tilde{b}^{2}_{i}\left[\left( \tilde{\gamma} + 2 \right) (1 - c_{2})^{2} - \dfrac{2}{\tilde{b}_{i}}(\tilde{b}_{i} - 1) + c_{1} C\right]y^{2}_{i} \right \rbrace,
\end{eqnarray}
for all $t\in \left(-1, 1 \right)$, where $c_{2} = \delta_{0}$. Since
\begin{eqnarray}
\label{c1 small and c2 almost 1}
\left( \tilde{\gamma} + 2 \right) (1 - c_{2})^{2} - \dfrac{2}{\tilde{b}_{i}}(\tilde{b}_{i} - 1) + c_{1} C & = & \nonumber  \tilde{\gamma}(1 - c_{2})^{2} + \left[2(1 - c_{2})^{2} - 2 + \dfrac{2}{\tilde{b}_{i}} \right] + c_{1} C \\ \nonumber
& \leq &  \tilde{\gamma}(1 - c_{2})^{2} + \left[2(1 - c_{2})^{2} - 2 + \dfrac{2}{\tilde{b}_{\max}}\right] + c_{1} C. \\ \nonumber
\end{eqnarray}
we can choose $0 < \tilde{\gamma} < 1/2 $ and $0 < c_{1} < 1/2$ depending on $\beta$ and $n$ to get
\begin{eqnarray}
\label{c1 small and c2 almost 1  2}
\left( \tilde{\gamma} + 2 \right) (1 - c_{2})^{2} - \dfrac{2}{\tilde{b}_{i}}(\tilde{b}_{i} - 1) + c_{1} C \leq (1 - c_{2})^{2} - 1 + \dfrac{1}{\tilde{b}_{\max}} < 0
\end{eqnarray}
and $\Vert  x + t y \Vert_{\tilde{\beta}} \geq c_{3}$, where $c_{3}=c_{3}(\beta,n)> 0$ is a constant. Thus, from \eqref{Bar. func. second derivat est.} it follows that
\begin{eqnarray}
\label{concavity}
G_{\gamma}''(t) & \leq & \nonumber \tilde{\gamma} \left[ \left( \dfrac{\tilde{c_{3}}^{(\tilde{\gamma}- 4)}\tilde{b}^{2}_{\min}}{4}\right)\left( (1 - c_{2})^{2} - 1 + \dfrac{1}{\tilde{b}_{\max}} \right) \right] \Vert y \Vert^{2}_{x} \\
& \leq & -\tilde{\gamma} c_{4} \Vert y \Vert^{2}_{x},
\end{eqnarray}
where $c_{4}= c_{4}(\beta,n)>0$ is a constant. We now use the mean value Theorem for difference quotients to obtain $t_{0} \in (-1, 1)$ such that
\begin{eqnarray}
\label{mean value for difference quotient}
G_{\gamma}(1) + G_{\gamma}(-1) - 2G_{\gamma}(0) = G_{\gamma}''(t_{0}).
\end{eqnarray}
Choose
$$
\vert x_{i_{0}} \vert \geq \left(\dfrac{1}{n} \right)^{2/\tilde{b}_{i_{0}}},
$$
and combine \eqref{integral on anisotropic cones via integral on classcal cones}-\eqref{geometric asymptotic condition on alpha}, \eqref{technical barrier integral estimate}, \eqref{concavity}, \eqref{mean value for difference quotient}, and Lemma \ref{technical lemma barrier} to find
\begin{eqnarray}
\int_{\mathcal{B}_{x_{0}}} \dfrac{y^{2}_{i_{0}}}{\Vert y \Vert^{c+2\alpha}} dy & \geq &  \nonumber  \sum_{i=1}^{n} \int_{\left(\Omega^{\tilde{\beta}, \delta_{0}}_{x_{0}}\cap \Theta_{r_{1}} \right)} \dfrac{m_{0}a_{ii} y^{2}_{i}}{\Vert y \Vert^{c+2\alpha}} dy + \sum_{i\neq j}^{} \int_{\left(\Omega^{\tilde{\beta}, \delta_{0}}_{x_{0}}\cap \Theta_{r_{1}}\right)} \dfrac{m_{0}a_{ij}y_{i}y_{j}}{\Vert y \Vert^{c+2\alpha}} dy \\ \nonumber
&\geq & \dfrac{c_{5}\vert \mathcal{S}^{\beta,n}_{} \vert}{2q_{\max}} - C_{6}\sum_{i\neq j}^{}\frac{1}{\left(\frac{2}{b_{i}} + \frac{2}{b_{j}} - \alpha \right)},
\end{eqnarray}
where $c_{5}=c_{5}(\beta, n)> 0$ and $C_{6} = C_{6}(\beta, n)> 1$ are constants. Letting $\alpha \rightarrow 2/ b_{\max}$, we obtain
\begin{eqnarray}
\label{barrier function inequality 2}
C_{\beta, \alpha} I_{1} \geq \tilde{\gamma} \dfrac{c_{4}c_{5}}{2}>0.
\end{eqnarray}
We now turn our efforts towards to estimate $I_{2}$ by below. Let us consider
\begin{eqnarray}
M := \dfrac{1}{\tilde{\gamma}}\sum_{i,j=1}^{n} \sup_{y \in \tilde{\Theta}_{1/4C_{0}}} \left | \partial^{2}_{ij}u_{\gamma}(x_{0} + y)\right|.
\end{eqnarray}
Note that $M >0$ does not depend on $\tilde{\gamma}$ and $x_{0}$. Let $r_{0} = r_{0}(\beta, n)  \in (0, 1)$ be a constant such that
\begin{eqnarray}
\label{barrier function inequality 3}
\tilde{\Theta}_{r_{0}} \subset \tilde{\Theta}_{1/4C_{0}} \quad \text{and} \quad \int_{\tilde{\Theta}_{r_{0}}} \dfrac{ M \vert y \vert^{2}}{ \Vert y \Vert^{c+s}} dy \leq \dfrac{c_{4}c_{5}}{8 C_{\beta, \alpha}}.
\end{eqnarray}
Moreover, we can write
\begin{eqnarray}
\label{barrier function inequality 4}
I_{2} & = & \displaystyle \int_{\left(\mathbb{R}^{n} \setminus \mathcal{B}_{x_{0}}\right) \cap \tilde{\Theta}_{r_{0}}} \dfrac{1 - u_{\gamma}\left( x_{0} + y\right)}{\Vert y \Vert^{c_{} + 2\alpha}} dy + \displaystyle \int_{\left(\mathbb{R}^{n} \setminus \mathcal{B}_{x_{0}}\right) \cap \left(\mathbb{R}^{n} \setminus \tilde{\Theta}_{r_{0}} \right)} \dfrac{1 - u_{\gamma}\left( x_{0} + y\right)}{\Vert y \Vert^{c_{} + 2\alpha}} dy \\ \nonumber
& = & I_{3} + I_{4}.
\end{eqnarray}
By symmetry we see that
\begin{eqnarray}
I_{3} = -\dfrac{1}{2}\displaystyle \int_{\left(\mathbb{R}^{n} \setminus \mathcal{B}_{x_{0}}\right) \cap \tilde{\Theta}_{r_{0}}} \dfrac{u_{\gamma}\left( x_{0} + y\right) + u_{\gamma}\left( x_{0} - y\right) - 2}{\Vert y \Vert^{c_{} + 2\alpha}} dy
\end{eqnarray}
and use Taylor's formula to get $0 \leq \theta^{i}_{y,x_{0}} \leq 1$ such that
\begin{eqnarray}
u_{\gamma}(x_{0} + y) + u_{\gamma}(x_{0} - y) - 2 = \langle D^{2}u_{\gamma}(x_{0} + \theta_{y,x_{0}}y)y, y \rangle,
\end{eqnarray}
where $\theta_{y,x_{0}}e_{i}=\theta^{i}_{y,x_{0}}$. Hence, from \eqref{barrier function inequality 3}, we get
\begin{eqnarray}
\label{barrier function inequality 5}
\left | I_{3} \right | \leq \tilde{\gamma} \dfrac{c_{4}c_{5}}{8C_{\beta, \alpha}}.
\end{eqnarray}
In order to estimate $I_{4}$ let $R_{0}=R_{0}(\beta, n) > 1$ a constant such that
$$
\tilde{\Theta}_{1}(-x_{0}) \cup \tilde{\Theta}_{r_{0}}\subset \tilde{\Theta}_{R_{0}}
$$
and denote $A_{0} = \left(\mathbb{R}^{n} \setminus \mathcal{B}_{x_{0}}\right) \cap \left(\mathbb{R}^{n} \setminus \tilde{\Theta}_{r_{0}} \right)$. Thus, since $u_{\gamma}\leq 1$ in $\mathbb{R}^{n} \setminus \tilde{\Theta}_{1}$, we find
\begin{eqnarray}
\label{barrier function inequality 6}
 I_{4} & = & \int_{A_{0} \cap \tilde{\Theta}_{R_{0}}} \dfrac{1 - u_{\gamma}(x_{0} + y)}{\Vert y \Vert^{c+\alpha}} dy + \int_{A_{0} \cap (\mathbb{R}^{n} \setminus \tilde{\Theta}_{R_{0}}) } \dfrac{1 - u_{\gamma}(x_{0} + y)}{\Vert y \Vert^{c+\alpha}} dy \\ \nonumber
& \geq  & \int_{A_{0} \cap \tilde{\Theta}_{R_{0}}} \dfrac{1 - u_{\gamma}(x_{0} + y)}{\Vert y \Vert^{c+\alpha}} dy.
\end{eqnarray}
Again we use $u_{\gamma}\leq 1$ in $\mathbb{R}^{n} \setminus \tilde{\Theta}_{1}$ to estimate
\begin{eqnarray}
\int_{A_{0} \cap \tilde{\Theta}_{R_{0}}} \dfrac{1 - u_{\gamma}(x_{0} + y)}{\Vert y \Vert^{c+\alpha}} dy &\geq & \int_{\left(A_{0} \cap \tilde{\Theta}_{R_{0}} \right) \cap \tilde{\Theta}_{1}(-x_{0})} \dfrac{1 - u_{\gamma}(x_{0} + y)}{\Vert y \Vert^{c+\alpha}} dy \\ \nonumber
& \geq & r_{0}^{-(c+2\alpha)}\int_{\left(A_{0} \cap \tilde{\Theta}_{R_{0}} \right) \cap \tilde{\Theta}_{1}(-x_{0})} \dfrac{1 - u_{\gamma}(x_{0} + y)}{\Vert y \Vert^{c+\alpha}} dy \\ \nonumber
& \geq & r_{0}^{-(c+4/b_{\max})}\int_{ \tilde{\Theta}_{1}(-x_{0})} \left(1 - u_{\gamma}(x_{0} + y)\right)dy, \\ \nonumber
\end{eqnarray}
since $r^{c+4/b_{\max}}_{0} \leq r_{0}^{c+2\alpha}$ and $u_{\gamma}\geq 1$ in $\tilde{\Theta}_{1}$.
Furthermore, we have
\begin{eqnarray}
\int_{ \tilde{\Theta}_{1}(-x_{0})} \left(1 - u_{\gamma}(x_{0} + y)\right)dy & = & \int_{ \tilde{\Theta}_{1}} \left(1 - u_{\gamma}(y)\right)dy \\ \nonumber
& = & c_{7}\left( \dfrac{1}{c_{\tilde{\beta}}} - \dfrac{1}{c_{\tilde{\beta}} - \tilde{\gamma}}\right)\\
& \geq &  \tilde{\gamma}\left[ - \dfrac{c_{7}}{c_{\tilde{\beta}}(c_{\tilde{\beta}}-1/2)} \right],
\end{eqnarray}
for some constant $c_{7}=c_{7}\left( \beta, n \right) >0$, since $0 < \tilde{\gamma} < 1/2$. Hence, letting $\alpha \rightarrow 2/b_{\max}$ we conclude that
\begin{eqnarray}
\label{barrier function inequality 7}
C_{\beta, \alpha} I_{4} \geq - \tilde{\gamma}\dfrac{c_{4}c_{5}}{8}.
\end{eqnarray}
We then combine \eqref{barrier function inequality 1}, \eqref{barrier function inequality 2}, \eqref{barrier function inequality 4}, \eqref{barrier function inequality 5}, and \eqref{barrier function inequality 7} to reach
$$
\Delta^{\beta, \alpha} u_{\gamma}\left( x_{0} \right) \geq \tilde{\gamma} \dfrac{c_{4}c_{5}}{4},
$$
if $\alpha > \alpha_{0}$ for some constant $\alpha_{0}= \alpha_{0}\left( \beta, n \right) > 0$. Finally, by Lemma \ref{scaling lemma on the anisotropic sphere} we get $\gamma \in \mathcal{I}_{\beta \alpha}$.
\end{proof}

From now on we assume $0 < \alpha_{0} < 2/b_{\max}$ is as in previous lemma. In next lemma we prove that $\mathcal{I}_{\beta\alpha}$ is a non-empty and bounded set.
\begin{lemma}
\label{fund. sol. lemma 1}
If $ \alpha_{0} < \alpha < 2/b_{\max}$, the constant $\gamma_{}^{\ast} = \gamma_{\beta, \alpha}^{\ast}$ satisfies $0< \gamma_{}^{\ast} < c$.
\end{lemma}
\begin{proof}
By previous lemma we have $\gamma^{\ast} > 0$. To get $\gamma^{\ast} < c$ we argue as in \cite{CWW}. We will show that there is $0 < \varepsilon_{0} < 1$ depending on $\beta$, $\alpha$, and $n$ such that if $\gamma > c - \varepsilon_{0}$ and $u \in \mathcal{A}^{+}_{\gamma}$ then there is $x_{0} \in \partial \tilde{\Theta}_{1}$ with
\begin{eqnarray}
\label{contradiction ineq fund. sol. lemma 1 1}
\Delta^{\beta, \alpha}u(x_{0}) < 0,
\end{eqnarray}
which contradicts $u \in \mathcal{A}^{+}_{\gamma}$. In fact, given $u \in \mathcal{A}^{+}_{\gamma}$ we can suppose
\begin{eqnarray}
u(x_{0}) = \displaystyle\inf_{x \in \partial \tilde{\Theta}_{1}}u = 1.
\end{eqnarray}
By $-\gamma$-homogeneity of $u$ and Lemma \ref{scaling properties lemma} we have
\begin{eqnarray}
\Vert x \Vert^{\tilde{\gamma} }_{\tilde{\beta}} u(x) \geq u(x_{0}) = 1,
\end{eqnarray}
that is, $v \leq u$ in $\mathbb{R}^{n}\setminus \left\lbrace 0 \right\rbrace.$
Then, we estimate
\begin{eqnarray}
\label{contradiction ineq fund. sol. lemma 1 2}
\Delta^{\beta, \alpha}u(x_{0}) & = & C_{\beta, \alpha}\int_{\mathbb{R}^{n}} \dfrac{1 - u(x_{0} + y)}{\Vert y \Vert^{c + 2\alpha}} dy \\ \nonumber
& \leq & C_{\beta, \alpha} \int_{\mathbb{R}^{n}} \dfrac{1 - v(x_{0} + y)}{\Vert y \Vert^{c+2\alpha}} dy.
\end{eqnarray}
As in \cite{CWW}, we write
\begin{eqnarray}
\label{contradiction ineq fund. sol. lemma 1 3}
\int_{\mathbb{R}^{n}} \dfrac{1 - v(x_{0} + y)}{\Vert y \Vert^{n + 2\alpha}} dy = I_{1} + I_{2},
\end{eqnarray}
where integrals $I_{i}$ are given by
\begin{eqnarray*}
I_{1} = \int_{A_{1}} \dfrac{1 - v(x_{0} + y)}{\Vert y \Vert^{c+2\alpha}} dy, \ \ \text{and} \ \ I_{2} = \int_{A_{2}} \dfrac{1 - v(x_{0} + y)}{\Vert y \Vert^{c+2\alpha}} dy
\end{eqnarray*}
with
$$
A_{1} = \tilde{\Theta}_{\frac{1}{4C_{0}}} \cup \tilde{\Theta}_{\frac{1}{2C_{0}}}(-x_{0}) \quad \text{and}\quad A_{2} = \mathbb{R}^{n} \setminus A_{1}.
$$
We use Taylor's formula to estimate
\begin{eqnarray}
\label{contradiction ineq fund. sol. lemma 1 4}
\int_{\tilde{\Theta}_{\frac{1}{4C_{0}}}} \dfrac{1 - v(x_{0} + y)}{\Vert y \Vert^{c+2\alpha}} dy & = & -\dfrac{1}{2}\int_{\tilde{\Theta}_{\frac{1}{4C_{0}}}} \dfrac{v(x_{0} + y)+ v(x_{0} - y ) - 2}{\Vert y \Vert^{c+2\alpha}} dy \\ \nonumber
& \leq &  \dfrac{1}{2} \int_{\tilde{\Theta}_{\frac{1}{4C_{0}}}} \dfrac{\vert v(x_{0} + y)+ v(x_{0} - y ) - 2 \vert}{\Vert y \Vert^{c+2\alpha}} dy \\ \nonumber
& \leq & C_{1},
\end{eqnarray}
where $C_{1} = C_{1}(\beta, \alpha, n) >0$ is a constant. Furthermore, if $y \in \tilde{\Theta}_{\frac{1}{2C_{0}}}(-x_{0})$ we get $\Vert y \Vert_{} \geq C^{-1}_{2}$, for some $C_{2}>1$ depending $\beta$ and $n$. Moreover, $v > 1$ in $\tilde{\Theta}_{\frac{1}{2C_{0}}}$. Hence, we estimate
\begin{eqnarray}
\label{contradiction ineq fund. sol. lemma 1 5}
\int_{\tilde{\Theta}_{\frac{1}{2C_{0}}}(-x_{0})} \dfrac{1 - v(x_{0} + y)}{\Vert y \Vert^{c+2\alpha}} dy & = &
\int_{\tilde{\Theta}_{\frac{1}{2C_{0}}}(-x_{0})} \dfrac{1 - \Vert x_{0} + y \Vert_{\tilde{\beta}}^{-\tilde{\gamma}}}{\Vert y \Vert^{c+2\alpha}} dy \\ \nonumber
& \leq & C_{2}^{-c-2\alpha}\int_{\tilde{\Theta}_{\frac{1}{2C_{0}}}(-x_{0})} \left(1 - \Vert x_{0} + y \Vert_{\tilde{\beta}}^{-\tilde{\gamma}} \right) dy \\ \nonumber
& = & C_{3} \left[\dfrac{(1/2C_{0})^{\frac{c}{\mu_{\beta}}}}{\frac{c}{\mu_{\beta}}} - \dfrac{(1/2C_{0})^{\frac{(c - \gamma)}{\mu_{\beta}}}}{\frac{(c - \gamma)}{\mu_{\beta}}} \right].
\end{eqnarray}
Finally, we get
\begin{eqnarray}
\label{contradiction ineq fund. sol. lemma 1 6}
\int_{A_{2}} \dfrac{1 - v(x_{0} + y)}{\Vert y \Vert^{c+2\alpha}} dy &\leq & C_{4}\int_{(\mathbb{R}^{n} \setminus \tilde{\Theta}_{\frac{1}{4C_{0}}})} \dfrac{1}{\Vert y \Vert^{c+2\alpha}} dy \\ \nonumber
& = & C_{5}.
\end{eqnarray}
Combining \eqref{contradiction ineq fund. sol. lemma 1 2} - \eqref{contradiction ineq fund. sol. lemma 1 6} we obtain \eqref{contradiction ineq fund. sol. lemma 1 1} and the lemma is proved.
\end{proof}

According to the notation used in \eqref{tilde notation}, given $\alpha >0$ we will denote by
$$
\tilde{\alpha} = \dfrac{\alpha}{\mu_{\beta}}.
$$
The next result asserts that $\mathcal{I}_{\beta\alpha}$ is a open subset of $\left(0, c \right]$. Through Lemmas \ref{scaling properties lemma} and \ref{scaling lemma on the anisotropic sphere}, its proof follows as in \cite{CWW}. This fact drives us to the heart of the existence of fundamental solutions, namely, the existence of solution to the following anisotropic non-local Poisson equation
\begin{eqnarray}
\label{main supersolution previous}
\Delta^{\beta, \alpha} u (x) = \dfrac{1}{\Vert x \Vert_{\tilde{\beta}}^{\tilde{\gamma} + 2\tilde{\alpha}}}, \quad \text{for all} \ x \in \mathbb{R}^{n} \setminus \left\lbrace 0 \right\rbrace.
\end{eqnarray}
\begin{lemma}
\label{barrier lemma lemma 1}
Assume $ \alpha_{0} < \alpha < 2/b_{\max}$ and $0 < \gamma < \gamma^{\ast}$. Then there exists $u \in \mathcal{A}_{\gamma}^{+}$ such that
\begin{eqnarray}
\Delta^{\beta, \alpha} u (x) \geq \dfrac{1}{\Vert x \Vert_{\tilde{\beta}}^{\left( \tilde{\gamma} + 2\tilde{\alpha} \right)}}, \quad \text{for all} \ x \in \mathbb{R}^{n} \setminus \left\lbrace 0 \right\rbrace.
\end{eqnarray}
\end{lemma}

We are now able to prove the existence of solution of \eqref{main supersolution previous}. We follow closely the proof presented in \cite{CWW}.
\begin{proposition}
\label{main supersolution prop}
Suppose $ \alpha_{0} < \alpha < 2/b_{\max}$. If $0 < \gamma < \gamma_{}^{\ast}$, then there exists $u \in \mathcal{A}_{\gamma}^{+}$ such that
\begin{eqnarray}
\label{main supersolution}
\Delta^{\beta, \alpha} u (x) = \dfrac{1}{\Vert x \Vert_{\tilde{\beta}}^{\tilde{\gamma} + 2\tilde{\alpha}}}, \quad \text{for all} \ x \in \mathbb{R}^{n} \setminus \left\lbrace 0 \right\rbrace.
\end{eqnarray}
\end{proposition}
\begin{proof}
Let $g: \left(0, +\infty \right) \rightarrow \left(0, +\infty \right)$ given by
$$
g(t) = \dfrac{1}{t^{\tilde{\gamma} + 2 \tilde{\alpha}}}.
$$
Now consider the sequence
$$
g_{k}(t):= \left \{
\begin{array}{ll}
2^{\left(\tilde{\gamma} + 2 \tilde{\alpha} \right)k}, \ \text{ if } \ t < 2^{-k}, \\
\dfrac{1}{t^{\left(\tilde{\gamma} + 2 \tilde{\alpha} \right)}}, \ \text{ if } \ 2^{-k} \leq t < 2^{k}, \\
0, \ \text{ if } \  t \geq 2^{k}.
\end{array}
\right.
$$
From Theorem \ref{Lq estimates} there exists a unique solution $u_{k} \in  H^{\alpha, q}_{\beta}\left ( \mathbb{R}^{n} \right)$ to the equation
\begin{eqnarray}
\Delta^{\beta, \alpha}u_{k}(x) + 2^{-k}u_{k}(x) = g_{k}\left(\Vert x \Vert^{}_{\tilde{\beta}}\right) \quad \text{for all} \ x \in \mathbb{R}^{n}.
\end{eqnarray}
Since
\begin{eqnarray}
\sup_{x \in \tilde{\Theta}_{2^{k}}}g_{k}\left(\Vert x \Vert^{}_{\tilde{\beta}}\right) \leq 2^{\left(\tilde{\gamma} + 2 \tilde{\alpha} \right)} \quad \text{and} \quad g_{k}\left(\Vert x \Vert^{}_{\tilde{\beta}}\right) = 0 \quad \text{in} \ \mathbb{R}^{n} \setminus \tilde{\Theta}_{2^{k}},
\end{eqnarray}
we can choose $q \geq 2$ such that $q \alpha > c$ and $g_{k} \in L^{q}(\mathbb{R}^{n})$. From Theorem \ref{main theorem} we get $u_{k} \in W_{\beta}^{\alpha, q}(\mathbb{R}^{n})$ and
\begin{eqnarray}
\vert u_{k}(x) - u_{k}(y) \vert \leq C \Vert x - y \Vert^{\alpha - \frac{c}{q}}_{}, \quad \text{for all} \ x,y \in \mathbb{R}^{n},
\end{eqnarray}
where $C>0$ depends on $\Vert u_{k} \Vert_{W_{\beta}^{\alpha, q}(\mathbb{R}^{n})}$. From Corollary \ref{max princ 1} it follows that
\begin{eqnarray}
\label{nonnegative uk}
u_{k} \geq 0 \quad \text{in} \ \mathbb{R}^{n}.
\end{eqnarray}
From Lemma \ref{barrier lemma lemma 1} we find $\varphi \in \mathcal{A_{\gamma}^{+}}$ such that
\begin{eqnarray}
\label{supersolution w}
 \Delta^{\beta, \alpha}\varphi(x) \geq g\left( \Vert x \Vert_{\tilde{\beta}}^{} \right) \geq g_{k}\left(\Vert x \Vert^{}_{\tilde{\beta}}\right) \quad \text{for all} \ x \in \mathbb{R}^{n} \setminus \left\lbrace 0 \right\rbrace.
\end{eqnarray}
Then, we combine \eqref{nonnegative uk}, \eqref{supersolution w}, and Lemma \ref{max principle} to get
\begin{eqnarray}
\label{relation between uk and w}
 \vert u_{k} \vert \leq \varphi  \quad \text{for all} \ x \in \mathbb{R}^{n} \setminus \left\lbrace 0 \right\rbrace.
\end{eqnarray}
Hence, by Lemma \ref{C-regularity in general} and \eqref{relation between uk and w}, we deduce that there exist $0< \theta < 1$ and $C_{1}> 0$ depending on $\beta$, $\alpha$, and $n$ such that
\begin{eqnarray}
\label{holder estimate for u_{k}}
 \Vert u_{k} \Vert_{C^{\theta}_{\Vert \cdot \Vert}(K)} \leq  \left( \Vert \varphi \Vert_{\infty}(K)+ \Vert \varphi \Vert_{ L^{1} (\mathbb{R}^{n}, w)} + \Vert g \Vert_{L^{\infty}(\Omega)} \right)
\end{eqnarray}
for every compact set $K \subset \Omega \subset \mathbb{R}^{n} \setminus \left\lbrace 0 \right\rbrace$, where $\Omega$ is an open. Thus, combining Ascoli-Arzel\`a, \eqref{relation between uk and w}, and \eqref{holder estimate for u_{k}} it follows that $u_{k}$ converges locally uniformly (up to a subsequence) to a non-negative function $u\in C^{0}\left( \mathbb{R}^{n} \setminus \left\lbrace 0 \right\rbrace \right) \cap L^{1} (\mathbb{R}^{n}, w)$.
We now prove $u \in \mathcal{A}^{+}_{\gamma}$, that is,
\begin{eqnarray}
\label{homogen. of u}
u(x) = r^{\gamma}u(T_{\beta,r}x) \quad \text{for all} \ x \neq 0.
\end{eqnarray}
Let $\varphi_{k,r}(x) = r^{\gamma}u_{k}(T_{\beta,r}x)$. Given $r>0$ there exists a natural number $m$ such that
\begin{eqnarray}
\label{}
2^{-m} \leq r^{\mu_{\beta}}, \ r^{2\alpha \mu_{\beta}} < 2^{m}.
\end{eqnarray}
Therefore, we see that
\begin{eqnarray}
\Delta^{\beta, \alpha}\varphi_{k,r}(x) + 2^{\left(-k-m \right)}\varphi_{k,r}(x) & = & \nonumber r^{\gamma + 2\alpha}\Delta^{\beta, \alpha}u_{k}(T_{\beta,r}x) + 2^{\left(-k-m \right)}(r^{\gamma})u_{k}(T_{\beta,r}x) \\ \nonumber
& \leq & r^{\gamma + 2\alpha}\left[ \Delta^{\beta, \alpha}u_{k}(T_{\beta,r}x) + 2^{-k}u_{k}(T_{\beta,r}x) \right]\\ \nonumber
& = & r^{\gamma + 2\alpha}g_{k}(r^{\mu_{\beta}} \Vert x \Vert_{\tilde{\beta}}).
\end{eqnarray}
We then find
\begin{eqnarray}
\Delta^{\beta, \alpha}\varphi_{k,r}(x) + 2^{\left(-k-m \right)}\varphi_{k,r}(x) & \leq & \nonumber g_{k+m}(\Vert x \Vert_{\tilde{\beta}})\\
& = & \Delta^{\beta, \alpha}u_{k+m}(x) + 2^{\left(-k-m \right)}u_{k+m}(x).
\end{eqnarray}
From Corollary \ref{max princ 1}, we obtain
\begin{eqnarray}
r^{\gamma}u_{k}(T_{\beta,r}x) = \varphi_{k,r}(x) \leq u_{k+m}(x)
\end{eqnarray}
and letting $k \rightarrow +\infty$ we find
\begin{eqnarray}
r^{\gamma}u_{}(T_{\beta,r}x)\leq u_{}(x).
\end{eqnarray}
Analogously, we have
\begin{eqnarray}
\Delta^{\beta, \alpha}\varphi_{k,r}(x) + 2^{\left(k-m \right)}\varphi_{k,r}(x) & \geq &
\Delta^{\beta, \alpha}u_{k-m}(x) + 2^{\left(k-m \right)}u_{k-m}(x),
\end{eqnarray}
that is,
\begin{eqnarray}
r^{\gamma}u_{}(T_{\beta,r}x)\geq u_{}(x).
\end{eqnarray}
Hence, \eqref{homogen. of u} is proved. To show $u$ is a solution of \eqref{main supersolution}, it suffices, see Lemma \ref{scaling lemma on the anisotropic sphere}, to prove that
\begin{eqnarray}
\label{u is a solution}
\Delta^{\beta, \alpha}u(x) = 1, \quad \text{for all} \ x \in \partial \tilde{\Theta}_{1}.
\end{eqnarray}
The strategy to get \eqref{u is a solution} is to apply Lemma \ref{stability prop}. However, the sequence $u_{k}$ is not uniformly bounded at origin. We overcome this difficulty by considering the following sequence
\begin{eqnarray}
\tilde{u}_{k} = u_{k}(1 - \nu),
\end{eqnarray}
where $\nu: \mathbb{R}^{n} \rightarrow \mathbb{R}^{n}$ is an anisotropic cut-off function such that
\begin{eqnarray}
\nu(x)= \left \{
\begin{array}{ll}
0, \ \text{ if } \ \Vert x \Vert_{\tilde{\beta}} \geq 1/16, \\
1, \ \text{ if } \ \Vert x \Vert_{\tilde{\beta}} < 1/16, \\
\end{array}
\right.
\end{eqnarray}
see Lemma 2.16 in \cite{LEITAO LP}. From \eqref{holder estimate for u_{k}} we deduce that $(\tilde{u}_{k})$ satisfies the conditions of Lemma \ref{stability prop} and, given $x_{0} \in \partial \tilde{\Theta}_{1}$, $\tilde{u}_{k}$ converges locally uniformly in $\tilde{\Theta}_{1/4C_{0}}(x_{0})$ to the function $u(1 -\nu)$. Notice that $\tilde{u}_{k}$ satisfies
\begin{eqnarray}
\Delta^{\beta, \alpha}\tilde{u}_{k}(x) = g_{k}\left(\Vert x \Vert^{}_{\tilde{\beta}}\right) - 2^{-k}u_{k}(x) + h_{k}(x) \quad \text{for all} \ x \in \tilde{\Theta}_{1/4C_{0}}(x_{0}),
\end{eqnarray}
where $h_{k}: \mathbb{R}^{n} \rightarrow \mathbb{R}^{n}$ is given by
$$
h_{k}(x) = \displaystyle\int_{\tilde{\Theta}_{1/16}(-x_{0})} \dfrac{u_{k}(x+y)}{\Vert y \Vert^{c+2\alpha}}dy.
$$
Clearly, $h_{k}$ converges locally uniformly in $\tilde{\Theta}_{1/4C_{0}}(x_{0})$ to function
$$
h(x) = \displaystyle\int_{\tilde{\Theta}_{1/16}(-x_{0})} \dfrac{u_{}(x+y)}{\Vert y \Vert^{c+2\alpha}}dy,
$$
since $\Vert y \Vert$ is bounded from below and from above in $\tilde{\Theta}_{1/16}(-x_{0})$. Hence, from Lemma \ref{stability prop}, we find
\begin{eqnarray}
\label{to deduce that u is solution}
\Delta^{\beta, \alpha}\left[ u - \nu u \right](x) =  g_{}\left(\Vert x \Vert^{}_{\tilde{\beta}}\right)  + h_{}(x), \quad \text{for all} \ x \in \tilde{\Theta}_{1/4C_{0}}(x_{0}).
\end{eqnarray}
In particular, if $x=x_{0}$ in \eqref{to deduce that u is solution}, we get \eqref{u is a solution}. Finally, by Lemma \ref{strong max principle}, we get $u \in \mathcal{A}^{+}_{\gamma}$ and the Lemma is proved.
\end{proof}

Through Lemmas \ref{scaling properties lemma} and \ref{scaling lemma on the anisotropic sphere}, we can prove, as in \cite{CWW}, that there does not exist solution $u \in \mathcal{A}^{+}_{\gamma^{\ast}}$ to the anisotropic non-local Poisson equation
$$
\Delta^{\beta, \alpha} u = \dfrac{1}{\Vert x \Vert_{\tilde{\beta}}^{\tilde{\gamma^{\ast}} + 2 \tilde{\alpha}}} \quad \text{for all} \ x \in \mathbb{R}^{n} \setminus \left\lbrace 0 \right\rbrace.
$$
In fact, we get:
\begin{lemma}
\label{barrier lemma lemma 2}
Let $ \alpha_{0} < \alpha < 2/b_{\max}$ and $0 < \gamma \leq \gamma^{\ast}$. If $u \in \mathcal{A}_{\gamma}^{+}$ satisfies
\begin{eqnarray}
\label{Poisson equation fails}
\Delta^{\beta, \alpha} u (x) \geq \dfrac{1}{\Vert x \Vert_{\tilde{\beta}}^{\tilde{\gamma} + 2\tilde{\alpha}}}, \quad \text{for all} \ x \in \mathbb{R}^{n} \setminus \left\lbrace 0 \right\rbrace,
\end{eqnarray}
then we get $\gamma < \gamma^{\ast}$.
\end{lemma}

\begin{remark}
\label{conseq. of main supersolution}
Assume that $ \alpha_{0} < \alpha < 2/b_{\max}$, $0 < \gamma < \gamma^{\ast}$ and $u \in \mathcal{A}^{+}_{\gamma}$. Via $C^{\theta}_{\Vert \cdot \Vert_{\tilde{\beta}}}$-estimates and a normalization argument, note that Lemma \ref{barrier lemma lemma 2} and Proposition \ref{main supersolution prop} allow us to let $\gamma \rightarrow \gamma^{\ast}$ in \eqref{main supersolution} and achieve our fundamental solution. Indeed, we have
\begin{eqnarray}
u(x) = \Vert x \Vert^{- \gamma} u \left( T^{-1}_{\beta, \Vert x \Vert_{}}x\right), \quad \text{for all} \ x \in \partial\Theta^{\tilde{\beta}}_{1}.
\end{eqnarray}
By the equivalence of anisotropic norms on anisotropic spheres, see assertion 3 of Lemma \ref{elemntary results}, we deduce that
\begin{eqnarray}
\Vert u_{\gamma} \Vert_{L^{\infty}\left(\partial \Theta^{\tilde{\beta}}_{1} \right)} \leq C \Vert u_{\gamma} \Vert_{L^{\infty}(\partial \Theta_{1})}
\end{eqnarray}
for some constant $C>0$ depending only on $\beta$ and $n$. Hence, if $u_{\gamma}$
satisfies \eqref{main supersolution}, we can use Lemmas \ref{barrier lemma lemma 2} and $C^{\theta}_{\Vert \cdot \Vert_{\tilde{\beta}}}$-estimates, see Lemma \ref{C-regularity in general}, to argue as in Proposition \ref{main supersolution prop} and find
\begin{eqnarray}
\label{conseq. of main supersolution 1}
\lim_{\gamma \rightarrow \gamma^{\ast}}\Vert u_{\gamma} \Vert_{L^{\infty}(\partial \Theta_{1})} = +\infty.
\end{eqnarray}
See Lemma 4.5 in \cite{CWW} and Section \ref{sec Existence of Fundamental solutions and consequences} for more details.
\end{remark}

The next result enables a comparison principle for $-\gamma$-homogeneous solutions. Its proof can be found in \cite{CWW} via Lemma \ref{strong max principle}.

\begin{lemma}
\label{strong maximum principle for homogeneous function}
If $0 < \gamma < c$ and $u, v \in \mathcal{A}_{\gamma}$ satisfy
\begin{eqnarray}
\Delta^{\beta, \alpha}u \leq 0 \leq \Delta^{\beta, \alpha}v \quad \text{in} \ \mathbb{R}^{n}\setminus \left\lbrace 0 \right\rbrace.
\end{eqnarray}
Then either $u =0$ or $v=0$ or $u = c_{0} v$ and $\Delta^{\beta, \alpha}u = 0 = \Delta^{\beta, \alpha}v$ in $\mathbb{R}^{n}\setminus \left\lbrace 0 \right\rbrace$, \\ for some constant $c_{0}>0$.
\end{lemma}

From previous Lemma and the existence of supersolution for $0 < \gamma < \gamma^{\ast}$ in $A^{+}_{\gamma}$, see Lemma \ref{barrier lemma lemma 2}, we get a maximum principle for subsolutions which guarantees the uniqueness of the fundamental solution (up to a constant).
\begin{corollary}
\label{barrier lemma corollary 1}
Let $ \alpha_{0} < \alpha < 2/b_{\max}$ and $0 < \gamma < \gamma^{\ast}$. If $u \in \mathcal{A}_{\gamma}$ is such that
\begin{eqnarray}
\Delta^{\beta, \alpha} u (x) \leq 0, \quad \text{for all} \ x \in \mathbb{R}^{n} \setminus \left\lbrace 0 \right\rbrace,
\end{eqnarray}
then $u=0$ in $\mathbb{R}^{n} \setminus \left\lbrace 0 \right\rbrace$.
\end{corollary}

\section{Proof of Theorem \ref{main theorem}}

This section is devoted to the proof of Theorem \ref{main theorem} which ensures $C^{\theta}_{\Vert \cdot \Vert}$ for functions in $H_{\beta}^{\alpha, q}(\mathbb{R}^{n})$ for some $0 < \theta < 1$ when $c=n$.

\begin{proof}[Proof of Theorem \ref{main theorem}]
Given $x, y \in \mathbb{R}^{n}$ assume that $R = \Vert x - y \Vert$ is positive. Let $C> 0$ be a constant depending on $\beta$ and $n$ such that
\begin{eqnarray}
E_{R}(x) \cup E_{R}(y) \subset E_{CR}(x) \cap E_{CR}(y).
\end{eqnarray}
If $u \in W_{\beta}^{\alpha, q}(\mathbb{R}^{n})$ we estimate
\begin{eqnarray}
\label{main theorem 1.1}
\vert u(x) - u(y) \vert \leq \vert u(x) - u_{x, CR} \vert + \vert u_{x, CR} - u_{y, CR} \vert + \vert u_{y, CR} - u(y) \vert.
\end{eqnarray}
By Lemma \ref{decay of average lemma} and the continuity of function $x \rightarrow u_{x, R}$ we deduce that the sequence of functions
$$
v_{R}(x):= u_{x, R}, \quad \text{for} \ x \in \mathbb{R}^{n},
$$
converges uniformly in $\mathbb{R}^{n}$ to a continuous function $h$ as $R \rightarrow 0$. Moreover, by Lebesgue differentiation theorem for anisotropic balls, see Lemma 3 in \cite{CC} or Lemma 3.6 in \cite{CLU}, we find
\begin{eqnarray}
\label{main theorem 1}
u(x) = \lim_{R \rightarrow 0^{+}} u_{x, R} = h(x), \quad \text{for almost everywhere} \ x \in \mathbb{R}^{n},
\end{eqnarray}
and since $h$ is continuous, we get $u = h$ in $\mathbb{R}^{n}$. Lemma \ref{decay of average lemma} also provide
\begin{eqnarray}
\label{main theorem 2}
\vert u(x) - u_{x, CR} \vert \leq C_{1}\left[ u\right]_{q, \alpha q}R^{\theta} \quad \text{and} \quad \vert u(y) - u_{y, CR} \vert \leq C_{1}\left[ u\right]_{q, \alpha q}R^{\theta}.
\end{eqnarray}
Integrating
\begin{eqnarray*}
\vert u_{x, CR} - u_{y, CR} \vert \leq \vert u(z) - u_{x, CR}  \vert + \vert u(z) - u_{y, CR} \vert,
\end{eqnarray*}
on $ E_{CR}(x) \cap E_{CR}(y)$ and taking into account that
$$
E_{R}(x)\subset E_{CR}(x) \cap E_{CR}(y) \subset E_{CR}(x) \quad \text{and} \quad  E_{R}(y)\subset E_{CR}(x) \cap E_{CR}(y)\subset E_{CR}(x)
$$
we estimate
\begin{eqnarray}
\label{main theorem 4}
\vert u_{x, CR} - u_{y, CR} \vert & \leq & \nonumber \dashint_{ E_{CR}(x) \cap E_{CR}(y)}\vert u(z) - u_{x, CR}\vert \, dz +  \dashint_{ E_{CR}(x) \cap E_{CR}(y)}\vert u(z) - u_{y, CR}\vert \, dz \\
& \leq & C_{2} \left( \dashint_{ E_{CR}(x)}\vert u(z) - u_{x, CR}\vert \, dz + \dashint_{E_{CR}(y)}\vert u(z) - u_{y, CR}\vert \, dz \right).
\end{eqnarray}
From \eqref{main theorem 4} and H\" older Inequality it follows that
\begin{eqnarray}
\label{main theorem 5}
\vert u_{x, CR} - u_{y, CR} \vert  \leq  C_{3}\left[ u\right]_{q, \alpha q}R^{\theta}.
\end{eqnarray}
Combining \eqref{main theorem 1.1}, \eqref{main theorem 2}, and \eqref{main theorem 5} we obtain
\begin{eqnarray}
\label{main theorem 6}
\vert u(x) - u(y) \vert \leq \left( C_{4} \Vert u \Vert_{W_{\beta}^{\alpha, q}(\mathbb{R}^{n})} \right) \Vert x - y \Vert^{\theta}.
\end{eqnarray}
Finally, by Lemma \ref{elemntary results}, we have
\begin{eqnarray}
\label{main theorem 7}
\Vert u \Vert_{L^{\infty}\left( \mathbb{R}^{n}\right)} \leq C_{5} \Vert u \Vert_{W_{\beta}^{\alpha, q}(\mathbb{R}^{n})}.
\end{eqnarray}
Hence, from \eqref{main theorem 6}, \eqref{main theorem 7}, and the second assertion of Lemma \ref{Main. Lemma} we conclude that
$$
H_{r_{\beta^{}}}^{\alpha, q}(\mathbb{R}^{n}) \hookrightarrow W_{\beta}^{\alpha, q}(\mathbb{R}^{n}) \hookrightarrow C^{\theta}_{\Vert \cdot \Vert}.
$$
\end{proof}
Theorem \ref{main theorem} assures pointwise vanishing at infinity for functions in anisotropic Bessel spaces when $c = n$.

\begin{corollary}
\label{u goes to zero as x tends to infty lemma}
Let $2 \leq q < +\infty$ and $u \in H^{\alpha, q}_{\beta} \left( \mathbb{R}^{n} \right)$ for $0 < \alpha < 2/b_{\max}$. If $c=n$, we get
\begin{eqnarray}
\label{u goes to zero as x tends to infty}
\lim_{\Vert x \Vert \rightarrow +\infty } u(x) = 0.
\end{eqnarray}
\end{corollary}
\begin{proof}
By Theorem \ref{main theorem} there exists a constant $C_{}>1$ such that
\begin{eqnarray}
\label{u goes to zero as x tends to infty 1}
\vert u(x) - u(y) \vert \leq C_{} \Vert x - y \Vert^{\theta}, \quad \text{for all} \ x,y \in \mathbb{R}^{n}.
\end{eqnarray}
Let $C_{1}>0$ a constant satisfying $\vert \Theta_{r} \vert = C_{1}r^{c}$ and
\begin{eqnarray}
C_{2} = \dfrac{1}{2^{q+1}}\left[ \dfrac{\left(2^{q+1} C^{q}_{} \right)^{\frac{c}{q \theta}}}{C_{1}} \right]^{-1}.
\end{eqnarray}
Since $u \in L^{q}\left(\mathbb{R}^{n}\right)$ there exists $R > 1$ such that
\begin{eqnarray}
\label{lp integrability of u}
\int_{\mathbb{R}^{n} \setminus \Theta_{R}} \vert u(y) \vert^{q} dy  < C_{2}\varepsilon^{q + \frac{c}{ \theta}}.
\end{eqnarray}
Choose $R_{1} > 1$ depending on $R$ which satisfies
\begin{eqnarray}
\Theta_{1}(x) \subset \mathbb{R}^{n} \setminus \Theta_{R} \quad \text{if} \ \Vert x \Vert > R_{1}.
\end{eqnarray}
On the other hand, from \eqref{u goes to zero as x tends to infty 1} we have
\begin{eqnarray}
\label{u goes to zero as x tends to infty 2}
\dfrac{\vert u(x) \vert^{q}}{2^{q}} - C^{q}_{} \Vert x - y \Vert^{q\theta} \leq \vert u(y) \vert^{q}, \quad \text{for all} \ x,y \in \mathbb{R}^{n}.
\end{eqnarray}
Combining \eqref{lp integrability of u} and \eqref{u goes to zero as x tends to infty 2} it follows that
\begin{eqnarray}
\label{u goes to zero as x tends to infty 3}
\vert u(x) \vert^{q} < 2^{q} C^{q}_{} h^{q \theta} + 2^{p} \dfrac{C_{2}\varepsilon^{q + \frac{c}{ \theta}}}{C_{1}h^{c}}, \quad \text{for all} \ h \in (0, 1).
\end{eqnarray}
In particular, if $h = \left( \dfrac{\varepsilon^{q}}{2^{q+1}C_{}}\right)^{\frac{1}{q\theta}}$ in \eqref{u goes to zero as x tends to infty 3}, we find
\begin{eqnarray}
\vert u(x) \vert < \varepsilon, \quad \text{if} \ \Vert x \Vert > R_{1}.
\end{eqnarray}
\end{proof}

\section{Existence of Fundamental solutions and consequences}
\label{sec Existence of Fundamental solutions and consequences}
In this section we demonstrate Theorem \ref{existence of fundamental solutions} and point out some of its consequences.
\begin{proof}[Proof of Theorem \ref{existence of fundamental solutions}]

Let $(\gamma_{k})$ be a sequence in $(0, \gamma^{*})$ which converges to $\gamma^{\ast}$. By Proposition \ref{main supersolution prop} there is $w_{k} \in \mathcal{A}^{+}_{\gamma_{k}}$ such that
\begin{eqnarray}
\Delta^{\beta, \alpha} w_{k}(x) = \dfrac{1}{\Vert x \Vert_{\tilde{\beta}}^{\left( \tilde{\gamma}_{k} + 2\tilde{\alpha} \right)}}, \quad \text{for all} \ x \in \mathbb{R}^{n} \setminus \left\lbrace 0 \right\rbrace.
\end{eqnarray}
Consider the sequence
\begin{eqnarray}
\tilde{w}_{k} = \dfrac{w_{k}}{\Vert w_{k} \Vert_{L^{\infty}(\partial \Theta_{1})}}.
\end{eqnarray}
Note that $\tilde{w}_{k}$ satisfies
\begin{eqnarray}
\label{convergence condition of tilde{w}k}
\Vert \tilde{w}_{k} \Vert_{L^{\infty}(\partial \Theta_{1})} = 1 \quad \text{and} \quad \Delta^{\beta, \alpha} \tilde{w}_{k} (x) = f_{k}(x),
\end{eqnarray}
for all $x \in \mathbb{R}^{n} \setminus \left\lbrace 0 \right\rbrace$, where
\begin{eqnarray}
f_{k}(x) = \dfrac{1}{\Vert w_{k} \Vert_{L^{\infty}(\partial \Theta_{1})}} \left(\dfrac{1}{\Vert x \Vert_{\tilde{\beta}}^{\left( \tilde{\gamma}_{k} + 2\tilde{\alpha} \right)}}\right).
\end{eqnarray}
From Lemma \ref{C-regularity in general} and \eqref{convergence condition of tilde{w}k}, $\tilde{w}_{k}$ converges locally uniformly (up to a subsequence) to a non-negative function $\Psi_{\beta, \alpha}\in C^{0}\left( \mathbb{R}^{n} \setminus \left\lbrace 0 \right\rbrace \right)$. Furthermore, from Remark \ref{conseq. of main supersolution}, $f_{k}$ converges locally uniformly to $0$ in $\mathbb{R}^{n} \setminus \left\lbrace 0 \right\rbrace$. Hence, by Lemma \ref{stability prop}, $\Psi_{\beta, \alpha}$ satisfies
\begin{eqnarray}
\Delta^{\beta, \alpha} \Psi_{\beta, \alpha} = 0, \quad \text{in} \ \mathbb{R}^{n} \setminus \left\lbrace 0 \right\rbrace.
\end{eqnarray}
Given $x\neq 0$ and $r >0$ we have
\begin{eqnarray}
\label{tildewk is Agamma}
\tilde{w}_{k}(x) = r^{\gamma_{k}} \tilde{w}_{k}\left(T_{\beta,r}x \right).
\end{eqnarray}
Letting $k\rightarrow + \infty$ in \eqref{tildewk is Agamma} we conclude that $\Psi_{\beta, \alpha} \in \mathcal{A}_{\gamma^{\ast}}$. From \eqref{convergence condition of tilde{w}k}, we deduced that
$$
\displaystyle\sup_{\partial \Theta_{1}}\Psi_{\beta, \alpha} = 1,
$$
since $\Psi_{\beta, \alpha}, \tilde{w}_{k}$ are continuous in $\partial \Theta_{1}$ and $\tilde{w}_{k}$ converges uniformly to $\Psi_{\beta, \alpha}$ in $\partial \Theta_{1}$.
To prove the uniqueness, note that, from Lemma \ref{strong max principle}, we reach $\Psi_{\beta, \alpha} \in \mathcal{A}^{+}_{\gamma^{\ast}}$. Thus, by Lemma \ref{barrier lemma corollary 1}, if $u \in \mathcal{A}^{+}_{\gamma}$ is a harmonic function, we find $\gamma = \gamma^{\ast}$. Finally, from Lemma \ref{strong maximum principle for homogeneous function}, we get
$$
u = \mu_{0}  \Psi_{\beta, \alpha},
$$
for some constant $\mu_{0} >0$ depending on $\beta$, $\alpha$ and $n$.
\end{proof}
In the following corollary, we prove that the fundamental solution $\Psi_{\beta, \alpha}$ is comparable to $\Vert x \Vert^{-\gamma^{\ast}}$ except at origin. We also get $\Psi_{\beta, \alpha} \in C^{\frac{2}{b_{\max}}}_{\Vert \cdot \Vert}\left(\mathbb{R}^{n} \setminus \left\lbrace 0 \right\rbrace  \right)$.

\begin{corollary}
\label{bounds and regularity of fund sol}
Let $ \alpha_{0} < \alpha < 2/b_{\max}$. If $\Psi_{\beta, \alpha}$ is as in Theorem \ref{existence of fundamental solutions}, then $\Psi_{\beta, \alpha}$ satisfies:
\begin{enumerate}
\item There is a constant $m_{0}>0$ depending on $\beta$, $\alpha$, and $n$ such that
\begin{eqnarray}
\label{bounds from below and from above}
\dfrac{m_{0}}{\Vert x \Vert^{\gamma^{\ast}}} \leq \Psi_{\beta, \alpha}(x) \leq \dfrac{1}{\Vert x \Vert^{\gamma^{\ast}}}, \quad \text{for all} \ x \in \mathbb{R}^{n} \setminus \left\lbrace 0 \right\rbrace.
\end{eqnarray}
\item $\Psi_{\beta, \alpha} \in L^{1}\left( \mathbb{R}^{n}, w \right) \cap C^{\frac{2}{b_{\max}}}_{\Vert \cdot \Vert}\left(\mathbb{R}^{n} \setminus \left\lbrace 0 \right\rbrace  \right)$.
\end{enumerate}
\end{corollary}
\begin{proof}
Since $\Psi_{\beta, \alpha}$ is $-\gamma^{\ast}$-homogeneous we get
\begin{eqnarray}
\label{bounds and regularity of fund sol 1}
\Psi_{\beta, \alpha}(x)= \Psi_{\beta, \alpha}(T_{\beta, \Vert x \Vert^{-1}}x) \Vert x \Vert^{-\gamma^{\ast}}, \quad \text{if} \ x \neq 0.
\end{eqnarray}
On the other hand, by Theorem \ref{Harnack Inequality}, there exists a constant $m_{0}>0$ depending on $\beta$, $\alpha$, and $n$ such that
\begin{eqnarray}
\label{bounds and regularity of fund sol 2}
\displaystyle\inf_{\partial \Theta_{1}}\Psi_{\beta, \alpha}\geq m_{0} \, \displaystyle\sup_{\partial \Theta_{1}}\Psi_{\beta, \alpha} = m_{0}.
\end{eqnarray}
We then combine \eqref{bounds and regularity of fund sol 1} and \eqref{bounds and regularity of fund sol 2} to find \eqref{bounds from below and from above}, since $\displaystyle\sup_{\partial \Theta_{1}}\Psi_{\beta, \alpha} =1$. To prove the second part of theorem, note that the $-\gamma^{\ast}$-homogeneity of $\Psi_{\beta, \alpha}$ implies that $\Psi_{\beta, \alpha} \in L^{1}\left( \mathbb{R}^{n}, w \right)$. To access $C^{\frac{2}{b_{\max}}}_{\Vert \cdot \Vert}$-regularity for $\Psi_{\beta, \alpha}$ it is sufficient to prove that
\begin{eqnarray}
\Psi_{\beta, \alpha} \in C^{\frac{2}{b_{\max}}}_{\Vert \cdot \Vert}\left( \partial \Theta_{1} \right).
\end{eqnarray}
As in \cite{CWW} we use induction on $k$. By Lemma \ref{C-regularity in general} and Remark \ref{C regilarity original norm remark} we can choose $k \in \mathbb{N}$ such that
\begin{eqnarray}
\dfrac{1}{k} < \min \left\lbrace 2/b_{\max}, \alpha \right\rbrace \quad \text{and} \quad \Psi_{\beta, \alpha} \in C^{\frac{1}{k}\frac{2}{  b_{\max}}}_{\Vert \cdot \Vert}\left( \partial \Theta_{1} \right).
\end{eqnarray}
Assume that
\begin{eqnarray}
\label{induction hypothesis}
\Psi_{\beta, \alpha} \in C^{\frac{2m}{kb_{\max}}}_{\Vert \cdot \Vert}\left( \partial \Theta_{1} \right), \quad \text{for} \ m=1, \dots, k-1.
\end{eqnarray}
Let $G: \mathbb{R}^{n} \rightarrow \mathbb{R}^{n}$ defined by
\begin{eqnarray}
G(x)= \left \{
\begin{array}{ll}
\Psi_{\beta, \alpha}(x), \ \text{ if } \ \mathbb{R}^{n} \setminus E_{\frac{1}{C}}, \\
0, \ \text{ if } \ E_{\frac{1}{C}}, \\
\end{array}
\right.
\end{eqnarray}
where $C>0$ is a constant such that $E_{\frac{1}{16C}} \subset \Theta_{1}$. Given $0 < r < 1$ and $\nu \in \partial \Theta_{1}$, from \eqref{induction hypothesis}, there exists a constant $C_{1} > 0$ which does not depend on $r$ and $\nu$ such that
\begin{eqnarray}
\label{induction hypothesis holder regularity}
\left | \Delta_{r,\nu}G \right | \leq C_{1}, \quad \text{in} \ \mathbb{R}^{n} \setminus E_{\frac{1}{C}},
\end{eqnarray}
where $\Delta_{r,\nu}G$ is the anisotropic difference quotient defined by
\begin{eqnarray}
\Delta_{r,\nu}G(x): = \dfrac{G(x + T_{\beta, r}\nu) - G(x)}{r^{\frac{2m}{kb_{\max}}}}.
\end{eqnarray}
Let $x_{0} \in \partial \Theta_{1}$ and $C_{2}>0$ such that
\begin{eqnarray}
E_{\frac{1}{C}}(-x) \cap E_{\frac{1}{C}} = \emptyset \quad \text{for all} \ x \in \Theta_{\frac{1}{C_{2}}}(x_{0}).
\end{eqnarray}
Now we prove $\Delta^{\beta, \alpha}\left[\Delta_{r,\nu}G \right]$ is bounded in $\Theta_{\frac{1}{C_{2}}}(x_{0})$. Indeed, since $\Psi_{\beta, \alpha}(x)$ is harmonic in $\mathbb{R}^{n} \setminus \left\lbrace 0 \right\rbrace $ and
$$
\displaystyle \int_{E_{\frac{1}{C}}(-x)} \dfrac{1}{\Vert y \Vert^{c + 2\alpha}} dy < + \infty, \quad \text{for} \ x \in \Theta_{\frac{1}{C_{2}}}(x_{0}),
$$
we get
\begin{eqnarray}
\Delta^{\beta, \alpha} G(x) = \displaystyle \int_{M(x)} \dfrac{\Psi_{\beta, \alpha}(x+y)}{\Vert y \Vert^{c + 2\alpha}} dy \quad \text{in} \ \Theta_{\frac{1}{C_{2}}}(x_{0}),
\end{eqnarray}
where we have denote by $M(x) = \mathbb{R}^{n} \setminus E_{\frac{1}{C}}(-x)$. From \eqref{bounds from below and from above} it follows that
\begin{eqnarray}
\label{boundness of fundamental solution}
\Psi_{\beta, \alpha} \leq C_{3}, \quad \text{in} \ \mathbb{R}^{n} \setminus E_{\frac{1}{C}}
\end{eqnarray}
for some constant $C_{3}> 0$ depending on $\beta$, $\alpha$ and $n$. Thus, for $x \in \Theta_{\frac{1}{C_{2}}}(x_{0})$, we find
\begin{eqnarray}
\Delta^{\beta, \alpha}\left[\Delta_{r,\nu}G\right](x) & = & \dfrac{1}{r^{\frac{2m}{kb_{\max}}}}\int_{M(x) \cap M(x^{}_{r, \nu_{\beta}})} \dfrac{\Psi_{\beta, \alpha}(x + y + T_{\beta, r}\nu) - \Psi_{\beta, \alpha}(x + y)}{\Vert y \Vert^{c + 2\alpha}} dy \\ \nonumber
& + & \dfrac{1}{r^{\frac{2m}{kb_{\max}}}}\int_{M(x) \cap (\mathbb{R}^{n} \setminus M(x^{}_{r, \nu_{\beta}}))}  \dfrac{\Psi_{\beta, \alpha}(x + y + T_{\beta, r}\nu)}{\Vert y \Vert^{c + 2\alpha}} dy \\ \nonumber
& - & \dfrac{1}{r^{\frac{2m}{kb_{\max}}}} \int_{M(x^{}_{r, \nu_{\beta}})\cap (\mathbb{R}^{n} \setminus M(x))} \dfrac{\Psi_{\beta, \alpha}(x + y)}{\Vert y \Vert^{c + 2\alpha}} dy,
\end{eqnarray}
where $x^{}_{r, \nu_{\beta}} = x + T_{\beta, r}\nu$. Moreover, there exists a constant $C_{4} >0$ such that
\begin{eqnarray}
\label{boundness of M(x) minus M(x+y)}
\max \left\lbrace \left | M(x) \cap (\mathbb{R}^{n} \setminus M(x^{}_{r, \nu_{\beta}}))\right |, \left | M(x^{}_{r, \nu_{\beta}})\cap (\mathbb{R}^{n} \setminus M(x))\right | \right\rbrace \leq C_{4}r^{2/b_{\max}}.
\end{eqnarray}
Combining \eqref{boundness of fundamental solution}, \eqref{induction hypothesis holder regularity}, and \eqref{boundness of M(x) minus M(x+y)} we get
\begin{eqnarray}
\label{boundness of the difference quotient}
\left | \Delta^{\beta, \alpha}\left[\Delta_{r,\nu}G \right](x) \right| \leq C_{5} \quad \text{in} \ \Theta_{\frac{1}{C_{2}}}(x_{0}).
\end{eqnarray}
Then, by Lemma \ref{C-regularity in general}, we deduce that
\begin{eqnarray}
\label{1/k holder of the difference quotient}
\Delta^{\beta, \alpha}\left[\Delta_{r,\nu}G \right] \in   C^{\frac{2}{kb_{\max}}}_{\Vert \cdot \Vert}\left(\Theta_{\frac{1}{2C_{2}}}(x_{0}) \right).
\end{eqnarray}
Letting $r \rightarrow 0$ in \eqref{1/k holder of the difference quotient} we reach
$$
\Psi_{\beta, \alpha} \in C^{\frac{2(m+1)}{kb_{\max}}}_{\Vert \cdot \Vert}\left( \partial \Theta_{1} \right)
$$
and the corollary is proved.
\end{proof}

Finally, we consider the anisotropic versions of the arguments used in \cite{CWW}, see Section 5 in \cite{CWW}, to get a Liouville-type result and, as a consequence, we obtain a characterization of singularity isolated of a viscosity solution of $\Delta^{\beta, \alpha} u = 0$ in $\mathbb{R}^{n} \setminus \left\lbrace 0 \right\rbrace$. Indeed, we get:
\begin{corollary}[Liouville and characterization of singularities] Suppose $ \alpha_{0} < \alpha < 2/b_{\max}$. Then the following assertions hold:
\begin{enumerate}
\item If $u$ is a bounded viscosity solution to the equation $\Delta^{\beta, \alpha}u = 0$ in $\mathbb{R}^{n}$ and satisfies
$$
\displaystyle \liminf_{r \rightarrow 0} \left( \inf_{\partial \Theta_{r}}\dfrac{u}{\Psi_{\beta, \alpha}} \right) = 0,
$$
then $u$ is constant.
\item Assume that $u$ satisfies all hypotheses of Theorem \ref{existence of fundamental solutions}. If $u$ is  bounded from above or below in $\Theta_{1} \setminus \left\lbrace 0 \right\rbrace $ and $u \in L^{\infty}\left( \mathbb{R}^{n} \setminus \Theta_{1} \right)$, then $u = \kappa_{1}$, or $u = \kappa_{2} \Psi_{\beta, \alpha} + \kappa_{1}$ for some $\kappa_{1}, \kappa_{2} \in \mathbb{R}$.
\end{enumerate}

\end{corollary}



\end{document}